\documentclass[11pt]{amsart} 
\usepackage[lmargin=1in,rmargin=1in,tmargin=1in,bmargin=1in]{geometry}
\usepackage[ps,all,arc,rotate]{xy}
\usepackage{graphicx, float, epstopdf}
\usepackage{bbm}
\usepackage{color}
\usepackage[unicode,bookmarks=false]{hyperref}
\hypersetup{hidelinks}
\usepackage{centernot}
\usepackage{fancyhdr}
\usepackage{multirow}
\usepackage[utf8]{inputenc}
\usepackage{amsfonts,amssymb,amsmath,amsthm,mathrsfs}
\usepackage{graphics, setspace}
\usepackage{braket}
\usepackage{mathtools}
\usepackage{tikz}  
\usepackage{upgreek}
\usepackage{xcolor}
\usepackage{array,esint}

\numberwithin{equation}{section}
\numberwithin{figure}{section}
\allowdisplaybreaks[4]         
 
\newtheorem{lemma}{Lemma}[section]
\newtheorem{theorem}{Theorem}[section]
\newtheorem{proposition}{Proposition}[section]

\newtheorem{corollary}[lemma]{Corollary}
\theoremstyle{definition}

\newtheorem{remark}{Remark}[section]
\usepackage{booktabs}
\usepackage{array}
\usepackage{vcell}

\newcommand{\R}{\mathbb{R}}
\newcommand{\C}{\mathbb{C}}

\newcommand{\e}{\operatorname{e}}

\usepackage{needspace}
\newcommand{\BA}{\mathcal{B}_A}
\newcommand{\FA}{\mathcal{F}_A}
\newcommand{\Eset}{\mathcal{E}}
\newcommand{\Sr}{\mathcal{S}_r}
\newcommand{\epsU}{\varepsilon_U}
\newcommand{\PP}{\mathcal{P}}

\newcommand{\EE}{\mathcal{E}_f}
\newcommand{\ind}{\mathbf 1}

\newcommand{\eps}{\varepsilon}
\title[A folklore conjecture on multiplicative exponential sums]{A folklore conjecture on multiplicative exponential sums}
\author[N.~Robles]{Nicolas Robles}
\address{Quantum Sky, Project 137, 11950 Democracy Drive, Suite 650, Reston, VA 20190, USA}
\email{robles.nicolas.m@gmail.com}

\begin{document}

\begin{abstract}
We prove a folklore conjecture on exponential sums with multiplicative
coefficients: if $f$ is multiplicative with $|f(n)|\le1$, $(a,q)=1$,
$|\alpha-a/q|\le q^{-2}$ and $q\le N/(\log N)^2$, then
\[
 \sum_{n\le N}f(n)\e(n\alpha)\ll\frac N{\log N}+\frac N{\sqrt q}.
\]
Complete multiplicativity is not required. We also prove two related
conjectured bounds: the refinement in which $q$ is replaced by
$q(1+N|\alpha-a/q|)$, when this quantity is at most $N/(\log N)^2$, and
the corresponding bound for the sum with logarithmic weights. Finally,
we show that the bound can fail for multiplicative $f$ satisfying only
$|f(p)|\le A$ at primes and $\sum_{n\le x}|f(n)|^2\le A^2x$, for every
fixed $A>1$.
\end{abstract}

\maketitle

\section{Introduction}\label{sec:intro}

\subsection{Previous results in the literature}\label{sec:literature}

Put $\e(t)=\exp(2\pi i t)$ and
\[
 S_f(N,\alpha)=\sum_{n\le N}f(n)\e(n\alpha),\qquad L=\log N\quad(N\ge3).
\]
Throughout, a multiplicative function satisfies $f(1)=1$ and
$f(mn)=f(m)f(n)$ whenever $(m,n)=1$. For a fixed $A\ge1$, define
\begin{align}
 \BA&=\{f\text{ multiplicative}: |f(n)|\le A\text{ for every }n\ge1\},\label{eq:BA}\\
 \FA&=\{f\text{ multiplicative}: |f(p)|\le A\text{ for primes }p,\
 \sum_{n\le x}|f(n)|^2\le A^2x\text{ for every real }x\ge1\}.\label{eq:FA}
\end{align}
If $|f(n)|\le A$ for every $n$, then $|f(p)|\le A$ and
$\sum_{n\le x}|f(n)|^2\le A^2\lfloor x\rfloor\le A^2x$, so that
$\BA\subset\FA$. The inclusion is strict for every $A\ge1$: the
multiplicative function equal to $n^{1/4}$ when $n$ is a square and to
$0$ otherwise is unbounded, and it belongs to $\mathcal F_1$ because
the sum of its squares over $n\le x$ is $\sum_{m\le\sqrt x}m\le x$.
The class $\FA$
is that of Montgomery and Vaughan \cite[Section~1]{MV}. All bounds are uniform in $f$, which may depend on
the summation length and the phase.

An early result is Daboussi's theorem: if $f\in\mathcal B_1$ and
$\alpha$ is irrational, then $S_f(N,\alpha)=o(N)$ as $N\to\infty$; see
\cite[Theorem~1]{D} and \cite{DD,DD2}. Montgomery and Vaughan
\cite[Corollary~1]{MV} proved that
\begin{equation}\label{eq:MV-intro}
 |S_f(N,\alpha)|\ll_A \frac NL+\frac N{\sqrt R}(\log R)^{3/2}
\end{equation}
for $f\in\FA$, $(a,q)=1$, $|\alpha-a/q|\le q^{-2}$ and
$2\le R\le q\le N/R$. Their argument expands $\log n$ as a sum over
prime divisors. They also showed that the term $N/L$ cannot be
decreased, even for completely multiplicative functions in
$\mathcal B_1$ \cite[Section~7]{MV}. At $\alpha=a/q$ with $q\le N$,
their Theorem~1 gives
\begin{equation}\label{eq:MV-thm1}
 |S_f(N,a/q)|\ll_A \frac NL+\frac N{\sqrt{\varphi(q)}}
 +\sqrt{Nq}\,\big(\log(2N/q)\big)^{3/2},
\end{equation}
and the last term is $O(N/L)$ when $q\le N/(L^2(\log L)^3)$.

For $f\in\mathcal B_1$, Bachman \cite[Theorem~4]{Bachman} obtained
\begin{equation}\label{eq:Bachman-intro}
 |S_f(N,\alpha)|\ll \frac NL+\frac N{\sqrt{\varphi(q)}}
\end{equation}
under $q\le Q=N/L^3$ and $|\alpha-a/q|\le1/(qQ)$. He also replaced the
factor $(\log R)^{3/2}$ in \eqref{eq:MV-intro} by
$\sqrt{\log R\log\log R}$ for $3\le R\le q\le N/R$
\cite[Theorem~5]{Bachman}. Here and below the
approximating fractions are reduced. For $\alpha=a/q$ with $q$ prime,
Maier and Sankaranarayanan \cite[Theorem~1.2]{MS} gave another proof of
the bound $N/L+N/\sqrt q$ in the range
$q\le N/(L^2\log L\log\log L)$, for sufficiently large $N$. They
identify this case as a consequence of Bachman's work. For the
class $\FA$, it was shown in
\cite[Theorem~1.1 and Corollary~1.1]{Robles} that
\begin{equation}\label{eq:R-intro}
 |S_f(N,\alpha)|\ll_A \frac NL+\frac N{\sqrt{\varphi(r)B}},
 \qquad B=\max\{1,N|\beta|\},
\end{equation}
when $\alpha=a/r+\beta$ with $(a,r)=1$, $r\le Q_0=N/L^3$ and
$|\beta|\le1/(rQ_0)$, and that the second term in
\eqref{eq:MV-intro} can be replaced by
$NR^{-1/2}\sqrt{\log\log(3R)}$ for $R\ge3$.

For $f\in\mathcal B_1$, sharper bounds are conjectured. De la
Bret\`eche and Granville \cite[(1.2)]{dBG} state as a folklore conjecture that
\begin{equation}\label{eq:folklore-conj}
 S_f(N,\alpha)\ll \frac NL+\frac N{\sqrt q}
\end{equation}
for completely multiplicative $f\in\mathcal B_1$, where
$q\le N/L^{2+\eps}$ for a fixed $\eps>0$ and $|\alpha-a/q|<q^{-2}$.
Their study of the major
arcs led them to the refined conjecture \cite[(1.4)]{dBG}
\begin{equation}\label{eq:refined-conj}
 S_f(N,\alpha)\ll \frac NL+\frac N{\sqrt{q(1+N|\beta|)}},
 \qquad \beta=\alpha-\frac aq.
\end{equation}
Granville and Lamzouri \cite[(1.2)]{GL} state
\eqref{eq:folklore-conj} for every $f\in\mathcal B_1$, with
$|\alpha-a/q|\le1/(qN)$ and $q\le N/L^{2+\eps}$. In the larger class
$\FA$, Bordignon, Bortolotto and Kerr \cite[(2)]{BBK} record the
expectation that the second term of \eqref{eq:MV-intro} can be replaced
by $N/\sqrt R$. Both \cite{dBG} and \cite{GL} also describe the
functions for which $S_f(N,\alpha)$ is large. For completely
multiplicative $f\in\mathcal B_1$, the sum can be large on the major
arcs only if $f$ is pretentious, that is, close to $\chi(n)n^{it}$ for
a primitive character $\chi$ of conductor dividing $q$ and a real $t$
\cite[Theorem~1]{dBG}. Under the hypotheses of
\cite[Corollary~1.1]{GL} on the minor arcs, if
$|S_f(N,\alpha)|\ge cN/L$ for a fixed $c>0$, then $f(p)$ correlates
with $\e(-hp\alpha)$ on an interval $z<p\le2z$, where $z\asymp_cN$ and
$h$ is a positive integer with $h=O_c(1)$.

The conjecture \eqref{eq:folklore-conj} is known in most ranges of $q$.
When $|\alpha-a/q|\le1/(qN)$, it follows from \cite[Theorem~4]{dBG}
that $S_f(N,\alpha)\ll N/\sqrt q$ for completely multiplicative
$f\in\mathcal B_1$ and $q\le L^{2-\eps}$; see \cite[Section~1.4]{GL}.
In the same range, the extension from completely multiplicative to
multiplicative coefficients in \cite[Section~2.2]{GL} carries the factor
\[
 \prod_{p^2\mid q}\frac{p+1}{p-1},
\]
which equals $1$ for squarefree $q$. De la Bret\`eche and Granville
\cite[Proposition~3]{dBG} sketch a proof that, for completely
multiplicative $f\in\mathcal B_1$, the bound \eqref{eq:folklore-conj}
can fail only if
\begin{equation}\label{eq:dBG-range}
 L^{2+o(1)}\le q\le L^2(\log L)^{1+o(1)},\qquad
 \bigg|\alpha-\frac aq\bigg|\le\frac{\log q\log\log q}{N},
\end{equation}
and Granville and Lamzouri \cite[Section~1.4]{GL} summarize
\eqref{eq:folklore-conj} as proved outside the range $q=L^{2+o(1)}$,
and inside it when $\varphi(q)\gg q$. In the remaining cases the
second terms of \eqref{eq:Bachman-intro} and \eqref{eq:R-intro} exceed
those of \eqref{eq:folklore-conj} and \eqref{eq:refined-conj} by the
factor $\sqrt{q/\varphi(q)}$, which is unbounded.

\subsection{New results}\label{sec:results}

We replace $\varphi(r)$ by $r$ in \eqref{eq:R-intro} for $f\in\BA$, and
we show that this replacement fails in $\FA$ for every fixed $A>1$. Our
main result is as follows. It contains the conjecture
\eqref{eq:folklore-conj} and, for bounded functions, the bound expected
in \cite{BBK}.

\begin{theorem}\label{thm:global}
Let $f\in\BA$. Suppose $(a,q)=1$, $|\alpha-a/q|\le q^{-2}$, and
$2\le R\le q\le N/R$. Then
\begin{equation}\label{eq:global}
 |S_f(N,\alpha)|\ll A\bigg(\frac NL+\frac N{\sqrt R}\bigg).
\end{equation}
For $(a,q)=1$, $1\le q\le N$ and $|\alpha-a/q|\le q^{-2}$, one also has
\begin{equation}\label{eq:three-term}
 |S_f(N,\alpha)|\ll A\bigg(\frac NL+\frac N{\sqrt q}+\sqrt{Nq}\bigg).
\end{equation}
Consequently, if $1\le q\le N/L^2$, then
\begin{equation}\label{eq:folklore}
 |S_f(N,\alpha)|\ll A\bigg(\frac NL+\frac N{\sqrt q}\bigg).
\end{equation}
\end{theorem}

The implied constants are absolute. The hypotheses under which
\eqref{eq:folklore-conj} is stated in \cite{dBG} and in \cite{GL} both
imply $q\le N/L^2$ and $|\alpha-a/q|\le q^{-2}$, so that
\eqref{eq:folklore} with $A=1$ proves the folklore conjecture in both
forms. Complete multiplicativity is not required. The dependence on $A$
is linear, by the reduction to $\mathcal B_1$ in
Lemma~\ref{lem:linear-A}, and it cannot be improved; see
Section~\ref{sec:A-sharpness}.

Theorem~\ref{thm:global} follows from an estimate on the arcs of
\eqref{eq:R-intro}, in which $\varphi(r)$ is replaced by $r$.

\begin{theorem}\label{thm:aux}
Let $f\in\BA$, $Q_0=N/L^3$, and suppose
\begin{equation}\label{eq:aux-hyp}
 (a,r)=1,\qquad 1\le r\le Q_0,\qquad
 \alpha=\frac ar+\beta,\qquad |\beta|\le\frac1{rQ_0}.
\end{equation}
Writing $B=\max\{1,N|\beta|\}$, one has
\begin{equation}\label{eq:aux-main}
 |S_f(N,\alpha)|\ll A\bigg(\frac NL+\frac N{\sqrt{rB}}\bigg).
\end{equation}
\end{theorem}

The same estimate gives the refined conjecture \eqref{eq:refined-conj}
in the range $q\max\{1,N|\beta|\}\le N/L^2$.

\begin{theorem}\label{thm:refined}
Let $f\in\BA$, $(a,q)=1$, $q\ge1$, and $\beta\in\R$. Put
$B=\max\{1,N|\beta|\}$ and $D=qB$. Then
\begin{equation}\label{eq:displacement-global}
 |S_f(N,a/q+\beta)|\ll A\bigg(\frac NL+\frac N{\sqrt D}+\sqrt{ND}\bigg).
\end{equation}
If
\begin{equation}\label{eq:wider-local}
 D\le N/L^2,
\end{equation}
one has
\begin{equation}\label{eq:refined}
 |S_f(N,a/q+\beta)|\ll A\bigg(\frac NL+
 \frac N{\sqrt{q\max\{1,N|\beta|\}}}\bigg).
\end{equation}
Condition \eqref{eq:wider-local} is equivalent to $q\le N/L^2$ and
$|\beta|\le1/(qL^2)$. The estimate \eqref{eq:refined} also holds when
$q\le Q_2=N/L^2$ and $|\beta|\le1/(qQ_2)$. In the conclusions,
$\max\{1,N|\beta|\}$ may be replaced by $1+N|\beta|$.
\end{theorem}

These bounds cannot be improved in several respects, even for
completely multiplicative functions in $\mathcal B_1$; the examples are
given in Section~\ref{sec:sharpness}. The term $N/L$ cannot be
decreased, by the construction of Montgomery and Vaughan
\cite[Section~7]{MV}, and primitive characters show that neither
$N/\sqrt R$ in \eqref{eq:global} nor $N/\sqrt q$ in \eqref{eq:folklore}
can be replaced by a function of smaller order. The factor $B^{-1/2}$
in \eqref{eq:aux-main} cannot be replaced by $B^{-1/2-\eps}$ for any
fixed $\eps>0$, as in \cite[Proposition~8.1]{Robles}. For every fixed
$0<\eta<2$, the bound \eqref{eq:folklore} fails for some
$q\asymp N/L^{2-\eta}$. Finally, \eqref{eq:refined} requires a
restriction on $\beta$: it fails if \eqref{eq:wider-local} is replaced
by $|\beta|\le q^{-2}$ and an upper bound for $q$.

By partial summation, Theorem~\ref{thm:global} also gives, for
$f\in\BA$, $(a,q)=1$, $1\le q\le N$ and $|\alpha-a/q|\le q^{-2}$,
\[
 \sum_{n\le N}\frac{f(n)}n\e(n\alpha)
 \ll A\bigg(\log(2q)+\frac L{\sqrt q}\bigg),
\]
which for $A=1$ is the bound conjectured in \cite[(1.5)]{GL}; see
Corollary~\ref{cor:log-weights}.

The situation is different in the class $\FA$. The next result shows
that, for every $A>1$, the totient in \eqref{eq:R-intro} cannot be
replaced by $r$.

\begin{theorem}\label{thm:sharpness-summary}
For every fixed $A>1$, there are arbitrarily large odd squarefree $q$
and functions $f_q\in\FA$ such that
\begin{equation}\label{eq:counter-summary}
 \sum_{n\le x}|f_q(n)|^2\le x\quad(x\ge1),\qquad
 |f_q(p^k)|\le A\quad(p\text{ prime},\ k\ge1),
\end{equation}
and, for every sufficiently large real $N$ depending on $q$,
simultaneously for every $(a,q)=1$,
\begin{equation}\label{eq:counter-order}
 |S_{f_q}(N,a/q)|\asymp\frac N{\sqrt{\varphi(q)}}.
\end{equation}
One may take $q/\varphi(q)\asymp\log\log q$, and $N$ with $L\ge q$;
then $q\le N/L^2$ and $N/L\le N/q$.
\end{theorem}

Thus neither \eqref{eq:folklore} nor the bound expected in \cite{BBK}
holds uniformly in $\FA$ for any fixed $A>1$, and the factor
$\sqrt{\log\log(3R)}$ in \cite[Corollary~1.1]{Robles} cannot be removed
in that class. The
functions $f_q$ have modulus greater than $1$ at some primes. The
extension of Theorems~\ref{thm:global}-\ref{thm:refined} to
$\mathcal F_1$ is not addressed here; see Section~\ref{sec:conclusion}.

\subsection{Outline of the proofs}\label{sec:ideas}

We describe the proof of Theorem~\ref{thm:aux}. The starting point is
the argument of \cite{Robles}. The identity of Montgomery and Vaughan
\[
 L\,S_f(N,\alpha)=\sum_{p\le N}f(p)S_f(N/p,p\alpha)\log p+O_A(N)
\]
reduces the problem to sums over primes. We may assume $r<L^3$, since
otherwise \eqref{eq:MV-intro} suffices. After the small primes are
removed as in \cite{Robles}, it remains to bound
\begin{equation}\label{eq:TP-intro}
 T(P)=\sum_{P<p\le2P}f(p)S_f(N/p,p\alpha)\log p
\end{equation}
for dyadic $P\ge16(rB)^2$. For such $P$, the Cauchy-Schwarz inequality
and the Brun-Titchmarsh inequality \cite[Theorem~2]{MV3} for the primes
in intervals of length about $P/B$ and in the reduced residue classes
modulo $r$ lead,
with $U=\lfloor N/P\rfloor$ and $T=4B$, to the mean square
\begin{equation}\label{eq:energy-intro}
 \EE(U,r,T)=\frac1{\varphi(r)}
 \sum_{\substack{b\bmod r\\(b,r)=1}}\int_{-T/U}^{T/U}
 \max_{k\le U}\bigg|\sum_{n\le k}f(n)\e\big(n(b/r+\theta)\big)\bigg|^2\,d\theta.
\end{equation}
In \cite{Robles} the sum over $b$ is extended to all residues modulo
$r$, and the Carleson-Hunt inequality \cite{C,H} then gives
$\EE(U,r,T)\ll_AU/\varphi(r)$ when $U\ge rT$. This step uses no
property of $f$ other than its mean square, and for arbitrary
coefficients of modulus $1$ it cannot be improved, as the coefficients
$\e(-n/r)$ show. It is the source of the totient in \eqref{eq:R-intro}.

The new ingredient is the bound
\begin{equation}\label{eq:energy-goal}
 \EE(U,r,T)\ll A^2\frac Ur
\end{equation}
for $f\in\BA$, when $r$ and $T$ are at most suitable powers of
$\log U$; see Propositions~\ref{prop:energy-small}
and~\ref{prop:energy-large}.
When \eqref{eq:energy-goal} holds, $T(P)\ll_AN/\sqrt{rB}$. To prove \eqref{eq:energy-goal} we expand in
Dirichlet characters modulo $r$ in the variable $b$. By orthogonality,
the mean over $b$ of
$|\sum_{n\le y}f(n)\e(n(b/r+\theta))|^2$ is
$\sum_{\chi\bmod r}|C_\chi(y,\theta)|^2$, where
\[
 C_\chi(y,\theta)=\sum_{n\le y}f(n)\Gamma_\chi(n;r)\e(n\theta),
 \qquad
 \Gamma_\chi(n;r)=\frac1{\varphi(r)}
 \sum_{\substack{b\bmod r\\(b,r)=1}}\chi(b)\e(bn/r).
\]
Write $n=sm$, where every prime factor of $s$ divides $r$ and
$(m,r)=1$. Then $\Gamma_\chi(sm;r)=\overline{\chi(m)}\Gamma_\chi(s;r)$
and $f(sm)=f(s)f(m)$, so that
\begin{equation}\label{eq:smooth-intro}
 C_\chi(y,0)=\sum_{\substack{s\le y\\p\mid s\Rightarrow p\mid r}}
 f(s)\Gamma_\chi(s;r)\sum_{m\le y/s}f(m)\overline{\chi(m)}.
\end{equation}
The Gauss sums $\Gamma_\chi(s;r)$ can be evaluated, and the sum of
$\max_\chi|\Gamma_\chi(s;r)|/s$ over these $s$ is
$\ll_\nu r^{-1/2+\nu}$ for every $\nu>0$. Together with the estimates
for character sums described next, this gives $r$ in place of
$\varphi(r)$.

The inner sums in \eqref{eq:smooth-intro} are character sums of $f$.
By the results of Granville, Harper and Soundararajan \cite{GHS}
around Hal\'asz's theorem, recalled in
Propositions~\ref{prop:GHS-long} and~\ref{prop:GHS-short-input}, they
are small outside a bounded set of
exceptional characters, in mean square over $\chi$ and, in short
intervals, for each $\chi$. The exceptional set is the same for all
the sums that occur. These results are applied to a completely
multiplicative function $g$ with $f=g*h$, where $h$ is small on
average (Lemma~\ref{lem:companion}). The saving is a power of $\log U$, and this restricts $r$ and
$T$ in \eqref{eq:energy-goal}. For an exceptional character we use
$\sum_{v\bmod r}|\Gamma_\chi(v;r)|\le(4/\sqrt6)\sqrt r$
(Lemma~\ref{lem:Gauss-mass}) and the Carleson-Hunt inequality, which
give a contribution $O(A^2U/r)$;
primitive characters show that this order is attained.

The integral over $\theta$ is treated in two ways. For
$T\le(\log U)^{1/4}$, the mean-square estimate at $\theta=0$ is
extended to $|\theta|\le T/U$ by partial summation. For
$T\le(\log U)^8$, we use the short-interval estimate and Gallagher's
lemma \cite[Lemma~1]{G}, in the form of Lemma~\ref{lem:Gallagher}. In
both cases the maximum over $k$ is first
restricted to about $(r/\varphi(r))T$ equally spaced points. The sums
between consecutive points are treated as in \cite{Robles}; they are
short enough for the factor $r/\varphi(r)$ to be absorbed.

The bound \eqref{eq:energy-goal} is available only when $\log U$
exceeds fixed powers of $r$ and $B$. For the remaining blocks $P<p\le2P$ with $P\le N/r$
we use the estimate
\[
 T(P)\ll_A\frac N{\sqrt{\varphi(r)B}}+\sqrt{\frac{rNP}{\varphi(r)}}
\]
of \cite[(6.10)]{Robles}. Their number is $O(1+(rB)^{1/2-c})$ for an
absolute constant $c>0$, so the first terms have total $O_A(N)$ in
spite of the factor $\sqrt{r/\varphi(r)}$. The second terms form a
geometric sum of size $O_A(N/\sqrt{\varphi(r)})$. The blocks with
$P>N/r$ are treated as in \cite{Robles}. There are $O(L)$ blocks in
all, and Theorem~\ref{thm:aux} follows on dividing by
$L$. The dependence on $A$ comes from the fact that the restriction of
$f\in\BA$ to an initial segment is a linear combination of functions
in $\mathcal B_1$ whose coefficients have absolute values summing to at
most $A$ (Lemma~\ref{lem:linear-A}). Theorems~\ref{thm:global} and~\ref{thm:refined} follow from
Theorem~\ref{thm:aux} by comparing $a/q$ with the fraction given by
Dirichlet's theorem, as in \cite{Robles}.

For Theorem~\ref{thm:sharpness-summary}, let $\chi$ be a primitive
character modulo a product $q$ of consecutive odd primes. The integers
coprime to $q$ have density $\varphi(q)/q$, so the mean square of
$\chi$ is $\varphi(q)/q$. We take
$f_q(n)=\chi(n)(1+\delta)^{\omega_{\PP}(n)/2}$, where $\delta>0$ is
small and $\omega_{\PP}(n)$ is the number of distinct prime factors of
$n$ in a set $\PP$ of primes not dividing $q$. The mean-value bound in
Proposition~\ref{prop:GHS-short-input}(iii) allows $\PP$ to be chosen so that
$\sum_{n\le x}|f_q(n)|^2\le x$ for every $x\ge1$, while the mean value
of $(1+\delta)^{\omega_{\PP}(n)/2}$ has order $\sqrt{q/\varphi(q)}$.
The Chinese remainder theorem then shows that $|S_{f_q}(N,a/q)|$
exceeds $|S_\chi(N,a/q)|\asymp N/\sqrt q$ by this factor.

\subsection{Notation}\label{sec:notation}

Sums with real upper limits run over positive integers up to their
floor; upper limits below $1$ give empty sums. The letter $p$ denotes a
prime, a dyadic number is a power of two, and $\ind_{\mathcal E}$ is $1$
when the indicated condition holds and $0$ otherwise. A star on a sum
over $b\bmod r$ restricts it to $(b,r)=1$. Characters are Dirichlet
characters modulo the stated modulus, including imprimitive characters.
Implied constants without subscripts are absolute; subscripts indicate
the permitted dependence. For positive quantities, $X\asymp Y$ means
$X\ll Y\ll X$, and $X\sim Y$ means $X/Y\to1$ in the limit considered.

\subsection{Organization of the paper}\label{sec:organization}

Section~\ref{sec:inputs} collects the imported estimates and the new
lemmas: the results of \cite{GHS} and \cite{Robles} that we use, the
reductions to completely multiplicative functions and to $A=1$, the
Gauss sums $\Gamma_\chi(n;r)$, and the passage from equally spaced
points to the maximum. Section~\ref{sec:energy} proves the bound
\eqref{eq:energy-goal} in the two ranges of $T$.
Section~\ref{sec:proofs} contains the proofs of
Theorems~\ref{thm:global}-\ref{thm:refined} and of the bound with
logarithmic weights. Section~\ref{sec:sharpness} gives the examples
showing that these results cannot be improved, and examples on
refinements proposed in the literature,
Section~\ref{sec:MV-counter} proves
Theorem~\ref{thm:sharpness-summary}, and Section~\ref{sec:conclusion}
states some open questions.

\section{Technical lemmas}\label{sec:inputs}

For complex coefficients $c_n$, a multiplicative function $f$, a
Dirichlet character $\chi$ and real $U,x\ge1$, write
\begin{equation}\label{eq:M-def}
 M_c(U;\theta)=\max_{0\le k\le U}\bigg|\sum_{n\le k}c_n\e(n\theta)\bigg|,
 \qquad A_f(x,\chi)=\sum_{n\le x}f(n)\overline{\chi(n)}.
\end{equation}
The maximum is over integers $k$. When $c_n=f(n)$ we write
$M_f(U;\theta)$.

\subsection{Imported estimates}\label{sec:imported}

For a multiplicative function $u$, let $\Lambda_u$ be defined by
\[
 -\frac{F_u'(s)}{F_u(s)}=\sum_{n\ge1}\frac{\Lambda_u(n)}{n^s},
 \qquad F_u(s)=\sum_{n\ge1}\frac{u(n)}{n^s},
\]
in a half-plane of absolute convergence. As in
\cite[Section~1]{GHS}, the class $\mathcal C(\kappa)$ consists of the
multiplicative functions $u$ for which $|\Lambda_u(n)|\le\kappa\Lambda(n)$
for every $n$. A completely multiplicative function in $\mathcal B_1$,
and each of its twists by a Dirichlet character, belongs to
$\mathcal C(1)$.

The first result is a case of the pretentious large sieve
\cite[Theorem~1.9]{GHS}. It bounds character sums in mean square, with
a set of exceptional characters that does not depend on the length of
the sum.

\begin{proposition}\label{prop:GHS-long}
Let $U$ be sufficiently large, $\ell=\log U$, $r\le\ell^4$, and let $g$
be completely multiplicative with $|g(n)|\le1$. There is a set $\Eset$
of at most $63$ characters modulo $r$ such that, for every $v$ with
$U^{1/4}\le v\le U$,
\begin{equation}\label{eq:GHS-long}
 \bigg(\sum_{\chi\notin\Eset}|A_g(v,\chi)|^2\bigg)^{1/2}
 \ll v\ell^{-7/8}(\log\ell)^{15/8}.
\end{equation}
\end{proposition}

\begin{proof}
Apply \cite[Theorems~1.8 and~1.9]{GHS} with $X=U^{1/2}$, modulus $r$,
$\kappa=1$ and $J=63$. The requirement $r\le X^{1/4}=U^{1/8}$ follows
from $r\le\ell^4$, and the range $\sqrt X\le x\le X^2$ of those
theorems is $U^{1/4}\le x\le U$. With $Q=r\log X$, the bound of
\cite[Theorem~1.9]{GHS} for the left side of \eqref{eq:GHS-long} is
\[
 v\Big(\frac{\log Q}{\log v}\Big)^{7/8}
 \log\Big(\frac{\log v}{\log Q}\Big).
\]
Here $\log Q\le5\log\ell$ and $\log v\ge\ell/4$, and
$t^{7/8}\log(1/t)$ is increasing for small $t>0$, so this is bounded by
the right side of \eqref{eq:GHS-long}. The exceptional set in
\cite[Theorem~1.8]{GHS} is the same for every $v$ in the stated range.
\end{proof}

The next proposition records Hal\'asz's theorem in short intervals, a
bound for $L$-functions with the small primes removed, and a mean-value
estimate.

\begin{proposition}\label{prop:GHS-short-input}
The following estimates hold.
\begin{enumerate}
\item[\textup{(i)}] Let $u\in\mathcal C(1)$, let $x$ be sufficiently
large, $0<\delta<1/4$, and $z=x^{1-\delta}$. Then
\begin{equation}\label{eq:short-halasz}
 \bigg|\sum_{x<n\le x+z}u(n)\bigg|
 \ll z\bigg((1+M)e^{-M}+\frac{\log(x/z)+\log\log x}{\log x}\bigg),
\end{equation}
where $M$ is defined by
\begin{equation}\label{eq:short-M}
 e^{-M}\log x=\max_{|t|\le(x/z)\log x}
 \bigg|\sum_{n\ge1}\frac{u(n)}{n^{1+1/\log x+it}}\bigg|.
\end{equation}
\item[\textup{(ii)}] Let $K\ge1$ be fixed and let $\chi$ be a
nonprincipal character modulo $r$. If $X\ge10$, $y\ge r\log X$,
$\sigma\ge1+1/\log X$ and $|t|\le(\log X)^K$, then
\begin{equation}\label{eq:rough-L}
 |L_y(\sigma+it,\chi)|\ll_K1,
 \qquad L_y(s,\chi)=\prod_{p>y}\big(1-\chi(p)p^{-s}\big)^{-1}.
\end{equation}
\item[\textup{(iii)}] Let $\kappa>0$ be fixed and $u\in\mathcal C(\kappa)$.
If $x$ is sufficiently large in terms of $\kappa$, then
\begin{equation}\label{eq:Shiu}
 \sum_{n\le x}|u(n)|\ll_\kappa\frac x{\log x}
 \exp\bigg(\sum_{p\le x}\frac{|u(p)|}p\bigg).
\end{equation}
\end{enumerate}
\end{proposition}

\begin{proof}
Part~(i) is \cite[Corollary~1.4]{GHS} with $\kappa=1$, and part~(ii) is
\cite[Lemma~5.1]{GHS}. For part~(iii), put $g(n)=|u(n)|$. By
\cite[Section~1]{GHS}, $g(n)\le d_\kappa(n)$, where $d_\kappa(n)$ is
the coefficient of $n^{-s}$ in $\zeta(s)^\kappa$, so that
$g(p^j)\le d_\kappa(p^j)=\binom{\kappa+j-1}{j}$.
Differentiating $\sum_{j\ge0}d_\kappa(p^j)z^j=(1-z)^{-\kappa}$ gives
\[
 \sum_{j\ge2}\frac{j g(p^j)}{p^j}
 \le\frac\kappa p\Big((1-p^{-1})^{-\kappa-1}-1\Big)
 \ll_\kappa p^{-2}.
\]
Since $g(p)\le\kappa$, the Chebyshev bound
\cite[Corollary~2.6]{MV2} now yields
\[
 \sum_{p\le t}g(p)\log p\ll_\kappa t\quad(t\ge1),\qquad
 \sum_p\sum_{j\ge2}\frac{j g(p^j)\log p}{p^j}
 \ll_\kappa\sum_p\frac{\log p}{p^2}<\infty.
\]
Thus \cite[Theorem~2.14 and Corollary~2.15]{MV2}, applied to the
nonnegative multiplicative function $g$, give, for $x\ge2$,
\[
 \sum_{n\le x}g(n)
 \ll_\kappa\frac x{\log x}
 \prod_{p\le x}\Big(1+\sum_{j\ge1}\frac{g(p^j)}{p^j}\Big)
 \ll_\kappa\frac x{\log x}
 \exp\bigg(\sum_{p\le x}\frac{g(p)}p\bigg),
\]
where the last step uses $\log(1+t)\le t$ for $t\ge0$ and the first
prime-power bound. For short intervals and arithmetic progressions,
see \cite[Lemma~2.3]{GHS} and Shiu's theorem \cite[Theorem~1]{S}.
\end{proof}

We also use the Carleson-Hunt inequality \cite{C,H} and its
consequence for sums over residue classes, in the forms given in
\cite{Robles}.

\begin{proposition}\label{prop:sampling}
Let $c=(c_n)$ be a sequence of complex numbers supported on at most $V$
consecutive integers, put $E=\sum_n|c_n|^2$, and let
$M_c(\theta)$ be the maximum over integers $k$ of
$|\sum_{n\le k}c_n\e(n\theta)|$. Then
\begin{align}
 \int_0^1M_c(\theta)^2\,d\theta&\ll E,\label{eq:CH}\\
 \sum_{b\bmod r}M_c(\theta+b/r)^2&\ll(V+r)E\label{eq:fullsample}
\end{align}
for every integer $r\ge1$ and every real $\theta$. The implied
constants are absolute.
\end{proposition}

\begin{proof}
When the support is contained in $\{1,\ldots,V\}$, these are
\cite[Proposition~2.3 and (4.4)]{Robles}. A shift of the support
multiplies every sum $\sum_{n\le k}c_n\e(n\theta)$ by the same factor of
modulus $1$, and so does not change $M_c$.
\end{proof}

If $c$ is supported on $1\le n\le U$, then $M_c(\theta)=M_c(U;\theta)$.

The next lemma is Gallagher's lemma \cite[Lemma~1]{G} for coefficients
in a Hilbert space. The proof is that of the scalar case; we include it
because it is short.

\begin{lemma}\label{lem:Gallagher}
Let $(a_n)$ be a finitely supported sequence in a complex Hilbert
space. For $\delta>0$ and $h_0=1/(2\delta)$,
\begin{equation}\label{eq:Gallagher}
 \int_{-\delta}^{\delta}\bigg\|\sum_na_n\e(n\theta)\bigg\|^2\,d\theta
 \ll h_0^{-2}\int_{\R}\bigg\|\sum_{t<n\le t+h_0}a_n\bigg\|^2\,dt.
\end{equation}
\end{lemma}

\begin{proof}
Let $\Phi(t)=\sum_na_n\ind_{[n-h_0,n)}(t)$, so that
$\Phi(t)=\sum_{t<n\le t+h_0}a_n$. The Fourier transform of $\Phi$ at
$\theta$ is the product of $\sum_na_n\e(-n\theta)$ and a factor of
modulus $|\sin(\pi h_0\theta)|/(\pi|\theta|)$, which is at least
$2h_0/\pi$ for $|\theta|\le1/(2h_0)$. Plancherel's theorem, applied to
the coordinates of $\Phi$ in an orthonormal basis of the span of the
$a_n$, gives \eqref{eq:Gallagher}.
\end{proof}

Finally we record what is used from \cite{Robles}. Recall the sums
$T(P)$ from \eqref{eq:TP-intro}.

\begin{proposition}\label{prop:old-blocks}
Let $f\in\FA$, suppose that \eqref{eq:aux-hyp} holds, and put $D=rB$.
\begin{enumerate}
\item[\textup{(i)}] If $r\ge L^3$, then $S_f(N,\alpha)\ll_AN/L$.
\item[\textup{(ii)}] Suppose that $r<L^3$ and that $N$ is sufficiently
large. Then $D\le L^3$. Let $Y$ be the least power of $2$ with
$Y\ge16D^2$. Then
\begin{equation}\label{eq:reduction}
 L\,S_f(N,\alpha)=\sum_{\substack{Y\le P<N\\P\textup{ dyadic}}}T(P)+O_A(N),
\end{equation}
and for every dyadic $P$ with $Y\le P<N$, writing
$U=\lfloor N/P\rfloor$,
\begin{align}
 |T(P)|&\ll_A\frac N{\sqrt{\varphi(r)B}}
 +\sqrt{\frac{rNP}{\varphi(r)}},\label{eq:old-block}\\
 |T(P)|&\ll_AN\sqrt{\frac{\log\log(3U)}U}
 \qquad(U\le r).\label{eq:old-short}
\end{align}
\item[\textup{(iii)}] Under the hypotheses of \textup{(ii)}, let $P$ be
dyadic with $Y\le P<N$, let $U=\lfloor N/P\rfloor$ and
$h=P/\lceil B\rceil$, and let $c_1,\ldots,c_U$ be complex numbers. Put
$d_n=nc_n$. Then $h|\beta|U\le1$ and
\begin{align}
 \sum_{P<p\le2P}M_c(U;\alpha p)^2\log p
 \ll\frac1{\varphi(r)}\sum_{b\bmod r}^{*}\int_P^{2P}
 \big(&M_c(U;ab/r+\beta t)^2\notag\\
 &+h^2|\beta|^2M_d(U;ab/r+\beta t)^2\big)\,dt,\label{eq:positive-transfer}
\end{align}
with an absolute implied constant.
\end{enumerate}
\end{proposition}

\begin{proof}
(i) For bounded $N$ the assertion follows from
$|S_f(N,\alpha)|\le AN$, which holds by the Cauchy-Schwarz inequality
and \eqref{eq:FA}. For large $N$ we have $L^3\le r\le N/L^3$ and
$|\beta|\le1/(rQ_0)\le r^{-2}$, so \eqref{eq:MV-intro} with $R=L^3$
gives
\[
 S_f(N,\alpha)\ll_A\frac NL+\frac{N(3\log L)^{3/2}}{L^{3/2}}
 \ll_A\frac NL.
\]
(ii) If $B=1$, then $D=r<L^3$; otherwise
$D=rN|\beta|\le N/Q_0=L^3$. The identity \eqref{eq:reduction} is
\cite[(7.2)]{Robles}; it is obtained in the proof of
\cite[Theorem~1.1]{Robles} from \cite[Lemmas~3.1 and~3.3]{Robles}. The
estimates \eqref{eq:old-block} and \eqref{eq:old-short} are
\cite[(6.10) and (6.11)]{Robles}; they follow from
\cite[Proposition~6.2]{Robles} and hold for $P\ge16D^2$.

(iii) The inequality \eqref{eq:positive-transfer} is
\cite[(6.6)]{Robles}, and $h|\beta|U\le1$ is part of
\cite[(6.4)]{Robles}. Both are obtained in the proof of
\cite[Proposition~6.1]{Robles}, for arbitrary complex numbers $c_n$
and $P\ge16D^2$.
\end{proof}

\subsection{Two reductions}\label{sec:reductions}

The first lemma replaces $f$ by a completely multiplicative function,
to which the results of \cite{GHS} apply.

\begin{lemma}\label{lem:companion}
If $f\in\BA$, there are a completely multiplicative function $g$ with
$|g(n)|\le1$ and a multiplicative function $h$ such that $f=g*h$ and
\begin{equation}\label{eq:h-norm}
 \sum_{d\ge1}\frac{|h(d)|}{d^\sigma}\ll_{A,\sigma}1\qquad(\sigma>1/2).
\end{equation}
In particular, for $x\ge1$,
\begin{equation}\label{eq:h-tail}
 \sum_{d>x}\frac{|h(d)|}d\ll_Ax^{-1/4},\qquad
 \sum_{d\le x}|h(d)|\ll_Ax^{3/4}.
\end{equation}
\end{lemma}

\begin{proof}
Set $g(p)=f(p)/\max\{1,|f(p)|\}$ and extend $g$ completely
multiplicatively. The function $h$ is then determined by
\[
 h(p)=f(p)-g(p),\qquad
 h(p^k)=f(p^k)-g(p)f(p^{k-1})\quad(k\ge2).
\]
For a finite set of primes with $|f(p)|>1$, multiplicativity gives
$\prod_p|f(p)|\le A$. Since a product of numbers $1+b_p$ with
$b_p\ge0$ is at least $1+\sum_pb_p$, it follows that
\[
 \sum_p\big(|f(p)|-1\big)_+\le A-1.
\]
Moreover, $|h(p)|=(|f(p)|-1)_+$ and $|h(p^k)|\le2A$ for $k\ge2$. The
Euler product for the left side of \eqref{eq:h-norm} is at most
\[
 \exp\bigg(A-1+2A\sum_p\frac{p^{-2\sigma}}{1-p^{-\sigma}}\bigg)<\infty.
\]
Taking $\sigma=3/4$ and comparing powers proves \eqref{eq:h-tail}.
\end{proof}

The second lemma gives the dependence on $A$ in all our estimates: it
suffices to prove them for $A=1$.

\begin{lemma}\label{lem:linear-A}
Let $A\ge1$, $X\ge1$ and $f\in\BA$. There are finitely many functions
$g_j\in\mathcal B_1$ and real numbers $w_j$ with $\sum_j|w_j|\le A$
such that
\begin{equation}\label{eq:decomposition-A}
 f(n)=\sum_jw_jg_j(n)\qquad(n\le X).
\end{equation}
Consequently, if $\mathscr N$ is a seminorm on
$\C^{\lfloor X\rfloor}$, then
\begin{equation}\label{eq:seminorm-A}
 \mathscr N\big((f(n))_{n\le X}\big)
 \le A\sup_{g\in\mathcal B_1}\mathscr N\big((g(n))_{n\le X}\big),
\end{equation}
and in particular, for all complex numbers $c_n$,
\begin{equation}\label{eq:linear-A}
 \bigg|\sum_{n\le X}c_nf(n)\bigg|
 \le A\sup_{g\in\mathcal B_1}\bigg|\sum_{n\le X}c_ng(n)\bigg|.
\end{equation}
\end{lemma}

\begin{proof}
For each prime $p\le X$, put
\[
 m_p=\max\big(\{1\}\cup\{|f(p^k)|:p^k\le X\}\big).
\]
For each $p$ with $m_p>1$, choose a prime power attaining the maximum,
and let $n_0$ be the product of these prime powers. Multiplicativity
and the bound $|f(n_0)|\le A$ give $\prod_{p\le X}m_p=|f(n_0)|\le A$.
The integer $n_0$ need not be at most $X$.

Define $u(p^k)=f(p^k)/m_p$ for $p^k\le X$. For each vector
$\eps=(\eps_p)_{p\le X}$ of signs $\eps_p=\pm1$, let $g_\eps$ be the
multiplicative function with
\[
 g_\eps(p^k)=\eps_pu(p^k)\quad(p^k\le X),
\]
and $g_\eps(p^k)=0$ at all other prime powers, and put
\[
 w(\eps)=\prod_{p\le X}\frac{1+\eps_pm_p}2.
\]
Each $g_\eps$ belongs to $\mathcal B_1$. For $m\ge1$,
\[
 \sum_{\eps=\pm1}\frac{1+\eps m}2=1,\qquad
 \sum_{\eps=\pm1}\eps\,\frac{1+\eps m}2=m,\qquad
 \sum_{\eps=\pm1}\bigg|\frac{1+\eps m}2\bigg|=m.
\]
Taking products of these identities over $p\le X$ gives
\[
 f(n)=\sum_\eps w(\eps)g_\eps(n)\quad(n\le X),
 \qquad\sum_\eps|w(\eps)|=\prod_{p\le X}m_p\le A,
\]
which is \eqref{eq:decomposition-A}. The triangle inequality and the
homogeneity of $\mathscr N$ give
\[
 \mathscr N\big((f(n))_{n\le X}\big)
 \le\sum_\eps|w(\eps)|\,\mathscr N\big((g_\eps(n))_{n\le X}\big),
\]
which proves \eqref{eq:seminorm-A}. Applying this to
$\mathscr N((z_n))=|\sum_{n\le X}c_nz_n|$ gives \eqref{eq:linear-A}.
\end{proof}

\subsection{Gauss sums and character projections}\label{sec:gauss}

Write
\[
 H_r=\frac r{\varphi(r)},\qquad V_r=\sum_{d\mid r}d^{-1/2},\qquad
 \Sr=\{s\ge1:p\mid s\Longrightarrow p\mid r\}.
\]
For functions of the reduced residues modulo $r$ we use the norm
\[
 \|F\|_{2,r}^2=\frac1{\varphi(r)}\sum_{b\bmod r}^{*}|F(b)|^2,
\]
for which the characters modulo $r$ form an orthonormal basis. Define
\begin{equation}\label{eq:Gamma}
 \Gamma_\chi(n;r)=\frac1{\varphi(r)}\sum_{b\bmod r}^{*}\chi(b)\e(bn/r).
\end{equation}
For complex coefficients $c_n$ and $y\ge0$, the sum
$F_y(b,\theta)=\sum_{n\le y}c_n\e(n(b/r+\theta))$ has the expansion
\begin{equation}\label{eq:char-expansion}
 F_y(b,\theta)=\sum_{\chi\bmod r}\overline{\chi(b)}C_\chi(y,\theta),
 \qquad C_\chi(y,\theta)=\sum_{n\le y}c_n\Gamma_\chi(n;r)\e(n\theta).
\end{equation}
For a set $\Eset$ of characters modulo $r$, we write $P_\Eset$ for the
orthogonal projection onto the span of the conjugates of the characters
in $\Eset$, so that
$P_\Eset F_y(b,\theta)=\sum_{\chi\in\Eset}\overline{\chi(b)}C_\chi(y,\theta)$.
Unless stated otherwise, $c_n=f(n)$.

The first lemma evaluates the sum of the absolute values of the Gauss
sums $\Gamma_\chi(v;r)$ over a period.

\begin{lemma}\label{lem:Gauss-mass}
If $\chi$ is a character modulo $r$ of conductor $d$, then
\begin{equation}\label{eq:Gamma-mass}
 \sum_{v\bmod r}|\Gamma_\chi(v;r)|
 =\sqrt d\,2^{\omega(r)-\omega(d)}\le\frac4{\sqrt6}\sqrt r.
\end{equation}
Also $|\Gamma_\chi(n;r)|\le1$.
\end{lemma}

\begin{proof}
For a character $\psi$ modulo $m$, write
\[
 \tau(\psi)=\sum_{v\bmod m}\psi(v)\e(v/m),
\]
with the character and the Gauss sum modulo $1$ both equal to $1$. Let
$\chi^*$ be the primitive character modulo $d$ inducing $\chi$. By
\cite[Theorem~9.7]{MV2},
\begin{equation}\label{eq:induced-Gauss}
 |\tau(\chi^*)|=\sqrt d.
\end{equation}
Put $m=r/(n,r)$ and $n'=n/(n,r)$. By \cite[Theorem~9.12]{MV2},
$\Gamma_\chi(n;r)=0$ unless $d\mid m$, and for $d\mid m$
\[
 \Gamma_\chi(n;r)=\overline{\chi^*(n')}\,\chi^*(m/d)\,\mu(m/d)\,
 \frac{\tau(\chi^*)}{\varphi(m)}.
\]
Here $|\chi^*(n')|=1$, because $(n',m)=1$ and $d\mid m$. Call $m$
admissible if $d\mid m$ and $m/d$ is squarefree and coprime to $d$.
By \eqref{eq:induced-Gauss}, $|\Gamma_\chi(n;r)|=\sqrt d/\varphi(m)$ if
$m$ is admissible, and $\Gamma_\chi(n;r)=0$ otherwise.
For each $m\mid r$ there are $\varphi(m)$ residue classes
$n$ modulo $r$ with $r/(n,r)=m$, so each admissible $m$ contributes
$\sqrt d$ to the sum in \eqref{eq:Gamma-mass}. The admissible $m$ are
the products $dk$, where $k$ is a product of distinct primes dividing
$r$ but not $d$; there are $2^{\omega(r)-\omega(d)}$ of them. This
gives the equality in \eqref{eq:Gamma-mass}.

For $p^\nu\Vert r$ and $p^\delta\Vert d$ with $\delta>0$, the local
factor of the ratio of $\sqrt d\,2^{\omega(r)-\omega(d)}$ to $\sqrt r$
is $p^{(\delta-\nu)/2}\le1$. If $\delta=0$, it is $2p^{-\nu/2}$. Only
$p=2$ and $p=3$ can give a factor larger than $1$, and the product of
these two factors is at most $4/\sqrt6$. The constant $4/\sqrt6$ is
attained by the principal character modulo $6$, for which the sum in
\eqref{eq:Gamma-mass} is $4$. The underlying bound
$2^{\omega(n)}\le(4/\sqrt6)\sqrt n$ is also used in \cite{GL}, after
(7.22) there. The bound $|\Gamma_\chi(n;r)|\le1$ is immediate from
\eqref{eq:Gamma}.
\end{proof}

The next lemma bounds the contribution of a single character, and of a
bounded set of characters, to the mean square \eqref{eq:energy-intro}.

\begin{lemma}\label{lem:one-character}
Let $U\ge1$, let $|c_n|\le A$ for $1\le n\le U$, and let $C_\chi$ and
$F_y$ be as in \eqref{eq:char-expansion}. Then
\begin{align}
 \sup_{y\le U,\,\theta\in\R}|C_\chi(y,\theta)|
 &\ll A\frac{U+r}{\sqrt r},\label{eq:one-point}\\
 \int_I\max_{y\le U}|C_\chi(y,\theta)|^2\,d\theta
 &\ll A^2\frac{U+r}r\label{eq:one-energy}
\end{align}
for every real interval $I$ of length at most $1/r$. If $\Eset$ is a
set of $J$ characters modulo $r$, then for every such interval $I$ one has
\begin{equation}\label{eq:finite-rank-energy}
 \frac1{\varphi(r)}\sum_{b\bmod r}^{*}\int_I
 \max_{y\le U}|P_\Eset F_y(b,\theta)|^2\,d\theta
 \ll A^2J^2\frac{U+r}r.
\end{equation}
\end{lemma}

\begin{proof}
Group $n$ according to its residue $v$ modulo $r$. Each group has at
most $U/r+1$ terms, and \eqref{eq:one-point} follows from
\eqref{eq:Gamma-mass}. For the integral, the sum over the $n\le y$ in one group
is $\e(v\theta)\sum_kc_{v+rk}\e(rk\theta)$, where $k$ runs over the
integers with $1\le v+rk\le y$. Its modulus is a function of
$\phi=r\theta$ of period $1$, and $I$ corresponds to an interval of
$\phi$ of length at most $1$. By \eqref{eq:CH}, applied in the variable
$\phi$, and $d\theta=d\phi/r$, the integral over $I$ of the square of
its maximum over $y\le U$ is $\ll A^2(U/r+1)/r$. Minkowski's inequality and \eqref{eq:Gamma-mass}
then give \eqref{eq:one-energy}. Finally, the Cauchy-Schwarz inequality
gives, for every $b$ and $\theta$,
\[
 \max_{y\le U}|P_\Eset F_y(b,\theta)|^2
 \le J\sum_{\chi\in\Eset}\max_{y\le U}|C_\chi(y,\theta)|^2,
\]
and \eqref{eq:finite-rank-energy} follows from \eqref{eq:one-energy}
on integrating over $I$ and averaging over $b$.
\end{proof}

The third lemma collects bounds for sums over $\Sr$. The estimate
\eqref{eq:W} supplies the factor $r^{-1/2}$ used in the character-sum
estimates of Section~\ref{sec:energy}.

\begin{lemma}\label{lem:smooth-weights}
For every $r\ge1$ and $x\ge1$,
\begin{align}
 &\sum_{s\in\Sr}s^{-1}=H_r,\qquad
 \sum_{s\in\Sr}s^{-1/2}\le r^2,\label{eq:smooth-sums}\\
 &\sum_{\substack{s\in\Sr\\s>x}}s^{-1}\le r^2x^{-1/2},\qquad
 \#\{s\in\Sr:s\le x\}\le r^2x^{1/2},\label{eq:smooth-tail}\\
 &W_r:=\sum_{s\in\Sr}\frac{\max_\chi|\Gamma_\chi(s;r)|}s
 \le\frac{H_rV_r}{\sqrt r}.\label{eq:W}
\end{align}
For fixed nonnegative $a,b$ and every $\nu>0$,
\begin{equation}\label{eq:subpower}
 H_r^aV_r^b\ll_{a,b,\nu}r^\nu.
\end{equation}
\end{lemma}

\begin{proof}
The assertions in \eqref{eq:smooth-sums} follow from Euler products and
$(1-p^{-1/2})^{-1}\le p^2$, and those in \eqref{eq:smooth-tail} follow
by comparing powers. For \eqref{eq:W}, group $s$ according to
$m=r/(s,r)$, and write $H_m=m/\varphi(m)$. The condition $r/(s,r)=m$
means that, for each prime $p\mid r$, the power of $p$ dividing $s$ is
exactly that dividing $r/m$ if $p\mid m$, and is at least that dividing
$r$ if $p\nmid m$. Hence the sum of $1/s$ over the $s\in\Sr$ with
$r/(s,r)=m$ is
\[
 \frac mr\prod_{\substack{p\mid r\\p\nmid m}}(1-p^{-1})^{-1}
 =\frac mr\frac{H_r}{H_m}.
\]
The proof of Lemma~\ref{lem:Gauss-mass} gives
$\max_\chi|\Gamma_\chi(s;r)|\le\sqrt m/\varphi(m)=H_m/\sqrt m$ whenever
$r/(s,r)=m$. Hence
\[
 W_r\le\sum_{m\mid r}\frac{H_m}{\sqrt m}\cdot\frac mr\frac{H_r}{H_m}
 =\frac{H_r}r\sum_{m\mid r}\sqrt m
 =\frac{H_rV_r}{\sqrt r}.
\]
Finally, the Euler factors of $H_r^aV_r^b$ are at most
$(1-p^{-1})^{-a}(1-p^{-1/2})^{-b}$. For all sufficiently large $p$
these are at most $p^\nu$, and the remaining primes contribute a
constant independent of their exponents. This proves
\eqref{eq:subpower}.
\end{proof}

Every positive integer has a unique decomposition $n=sm$ with
$s\in\Sr$ and $(m,r)=1$. Then $(s,m)=1$, so that $f(sm)=f(s)f(m)$ for
every multiplicative $f$, and
$\Gamma_\chi(sm;r)=\overline{\chi(m)}\Gamma_\chi(s;r)$. Consequently,
when $c_n=f(n)$,
\begin{equation}\label{eq:smooth-expansion}
 C_\chi(y,0)=\sum_{\substack{s\in\Sr\\s\le y}}
 f(s)\Gamma_\chi(s;r)A_f(y/s,\chi).
\end{equation}
For the sum over an interval $x<n\le y$, the factor $A_f(y/s,\chi)$ is
replaced by $A_f(y/s,\chi)-A_f(x/s,\chi)$.

\subsection{Equally spaced points}\label{sec:grid}

The last lemma bounds the maximum over all $y\le U$ by the maximum
over a set of points $y_j$, at the cost of the sums between
consecutive points. In these sums no projection is applied.

\begin{lemma}\label{lem:coarse-grid}
Let $c_n$ be supported on $1\le n\le U$, with $E=\sum_n|c_n|^2$, and
let $F_y$ be as in \eqref{eq:char-expansion}. Let $P$ be an orthogonal
projection on the space of functions of the reduced residues modulo
$r$, and put $R_y=(1-P)F_y$. For a real interval $I$ and a function
$\Psi(b,\theta)$, write
\[
 \|\Psi\|_{I,r}^2=\frac1{\varphi(r)}\sum_{b\bmod r}^{*}
 \int_I|\Psi(b,\theta)|^2\,d\theta.
\]
Let $0=y_0<y_1<\cdots<y_G=W$ be integers with $W\ge U$, and put
$\Delta=\max_j(y_{j+1}-y_j)$. Then
\begin{equation}\label{eq:coarse-grid-transfer}
 \big\|\max_{y\le U}|F_y|\big\|_{I,r}^2
 \le3\big\|\max_{0\le j\le G}|R_{y_j}|\big\|_{I,r}^2
 +3\big\|\max_{y\le W}|PF_y|\big\|_{I,r}^2
 +C\frac{|I|(\Delta+r)}{\varphi(r)}E,
\end{equation}
where $C$ is an absolute constant.
\end{lemma}

\begin{proof}
If $y_j\le y\le y_{j+1}$, then
\[
 F_y=R_{y_j}+PF_{y_j}+(F_y-F_{y_j}).
\]
Let $M_j(b,\theta)$ be the maximum over $y_j\le y\le y_{j+1}$ of
$|F_y(b,\theta)-F_{y_j}(b,\theta)|$. Then, for every $b$ and $\theta$,
\[
 \max_{y\le U}|F_y|
 \le\max_j|R_{y_j}|+\max_{y\le W}|PF_y|+\max_jM_j.
\]
The square of the right side is at most three times the sum of the
three squares, and $\max_jM_j^2\le\sum_jM_j^2$. For fixed $\theta$,
extend the sum over reduced residues of $\sum_jM_j^2$ to all residues
modulo $r$. Since $F_y-F_{y_j}$ involves only the coefficients $c_n$
with $y_j<n\le y_{j+1}$, the estimate \eqref{eq:fullsample} gives
\[
 \frac1{\varphi(r)}\sum_{b\bmod r}^{*}\sum_jM_j(b,\theta)^2
 \ll\frac{\Delta+r}{\varphi(r)}\sum_j\sum_{y_j<n\le y_{j+1}}|c_n|^2
 =\frac{\Delta+r}{\varphi(r)}E.
\]
Integrating over $I$ proves \eqref{eq:coarse-grid-transfer}.
\end{proof}

\section{Maximal mean-square estimates}\label{sec:energy}

In the notation \eqref{eq:M-def}, the mean square
\eqref{eq:energy-intro} is
\begin{equation}\label{eq:energy-def}
 \EE(U,r,T)=\frac1{\varphi(r)}\sum_{b\bmod r}^{*}
 \int_{-T/U}^{T/U}M_f(U;b/r+\theta)^2\,d\theta.
\end{equation}
Throughout this section $\ell=\log U$, and $F_y$, $C_\chi$ and
$P_\Eset$ are as in Section~\ref{sec:gauss}, with $c_n=f(n)$.

For fixed $U$, $r$ and $T$, the quantity $\EE(U,r,T)^{1/2}$ is a
seminorm of the vector $(f(n))_{n\le U}$. Indeed, the maximum in
\eqref{eq:M-def} is subadditive in the coefficients, and Minkowski's
inequality gives the triangle inequality after integration and
averaging. Lemma~\ref{lem:linear-A} therefore yields, for $f\in\BA$,
\begin{equation}\label{eq:energy-A-transfer}
 \EE(U,r,T)^{1/2}
 \le A\sup_{g\in\mathcal B_1}\mathcal E_g(U,r,T)^{1/2}.
\end{equation}
Thus it suffices to prove Propositions~\ref{prop:energy-small}
and~\ref{prop:energy-large} for $f\in\mathcal B_1$.

\subsection{Short arcs}\label{sec:small-energy}

The first estimate covers the range $T\le\ell^{1/4}$.

\begin{proposition}\label{prop:energy-small}
There is an absolute constant $U_0$ such that, if $f\in\BA$,
$U\ge U_0$, $r\le(\log U)^4$ and $1\le T\le(\log U)^{1/4}$, then
\begin{equation}\label{eq:energy-small}
 \EE(U,r,T)\ll A^2\frac Ur,
\end{equation}
with an absolute implied constant.
\end{proposition}

The proof rests on the following bound at $\theta=0$ for the part of
$F_y$ orthogonal to a bounded set of characters, which is the same for
all $y$.

\begin{lemma}\label{lem:long-remainder}
Let $f\in\BA$, $U\ge U_1(A)$ and $r\le\ell^4$. There is a set $\Eset$
of at most $63$ characters modulo $r$ such that $R_y=(1-P_\Eset)F_y$
satisfies
\begin{equation}\label{eq:long-remainder}
 \|R_y(\cdot,0)\|_{2,r}\ll_A\frac y{\sqrt r}\ell^{-3/4}
 \qquad(U^{1/2}\le y\le U).
\end{equation}
For $0\le y<U^{1/2}$ one has $\|R_y(\cdot,0)\|_{2,r}\le Ay$.
\end{lemma}

\begin{proof}
Let $g$ and $h$ be as in Lemma~\ref{lem:companion}, and let $\Eset$ be
the set given by Proposition~\ref{prop:GHS-long} for $g$. Write
$E_0=\ell^{-7/8}(\log\ell)^{15/8}$. Since $f=g*h$,
\[
 A_f(z,\chi)=\sum_{d\le z}h(d)\overline{\chi(d)}A_g(z/d,\chi).
\]
We claim that, uniformly for $U^{3/8}\le z\le U$,
\begin{equation}\label{eq:f-long-character}
 \bigg(\sum_{\chi\notin\Eset}|A_f(z,\chi)|^2\bigg)^{1/2}
 \ll_Az\big(E_0+\sqrt r\,U^{-1/32}\big).
\end{equation}
For $d\le U^{1/8}$ we have $U^{1/4}\le z/d\le U$, and Minkowski's
inequality, \eqref{eq:GHS-long} and \eqref{eq:h-norm} with $\sigma=1$
show that these $d$ contribute $\ll_AzE_0$. For $d>U^{1/8}$ we use
$|A_g(z/d,\chi)|\le z/d$ and \eqref{eq:h-tail}, and that there are at
most $r$ characters.

Now let $U^{1/2}\le y\le U$. By Parseval's identity for
\eqref{eq:char-expansion},
\[
 \|R_y(\cdot,0)\|_{2,r}^2=\sum_{\chi\notin\Eset}|C_\chi(y,0)|^2,
\]
and we apply \eqref{eq:smooth-expansion}. For $s\le U^{1/8}$ we have
$y/s\ge U^{3/8}$, so by Minkowski's inequality,
\eqref{eq:f-long-character} and \eqref{eq:W}, the terms with
$s\le U^{1/8}$ contribute to $\|R_y(\cdot,0)\|_{2,r}$ at most
\[
 C_A\frac y{\sqrt r}H_rV_r\big(E_0+\sqrt r\,U^{-1/32}\big).
\]
For $s>U^{1/8}$, the bounds $|A_f(y/s,\chi)|\le Ay/s$ and
$|\Gamma_\chi(s;r)|\le1$, followed by \eqref{eq:smooth-tail}, give a
contribution at most $C_Ayr^{5/2}U^{-1/16}$. By \eqref{eq:subpower}
with $\nu=1/128$ and $r\le\ell^4$, both expressions are at most
$C_Ay\ell^{-3/4}/\sqrt r$ for sufficiently large $U$. This proves
\eqref{eq:long-remainder}. Finally, an orthogonal projection does not
increase the norm and $|F_y(b,0)|\le Ay$, which gives the last
assertion.
\end{proof}

\begin{proof}[Proof of Proposition~\ref{prop:energy-small}]
By \eqref{eq:energy-A-transfer}, we may assume $f\in\mathcal B_1$. All
constants and all lower bounds for $U$ in this proof are then absolute.
Let $\Eset$ and $R_y$ be as in Lemma~\ref{lem:long-remainder}; the same
set $\Eset$ is used for every $y$ and $\theta$. Partial summation gives
\begin{equation}\label{eq:Abel-remainder}
 R_y(b,\theta)=\e(\theta y)R_y(b,0)
 -2\pi i\theta\int_0^y\e(\theta v)R_v(b,0)\,dv.
\end{equation}
We use \eqref{eq:long-remainder} for $v\ge U^{1/2}$ and the bound
$\|R_v(\cdot,0)\|_{2,r}\le v$ for $v<U^{1/2}$. The range $v<U^{1/2}$
contributes $O(T)$ to the norm of the second term in
\eqref{eq:Abel-remainder}, which is smaller than the right side of
\begin{equation}\label{eq:displaced-remainder}
 \|R_y(\cdot,\theta)\|_{2,r}\ll\frac U{\sqrt r}T\ell^{-3/4}
 \qquad(U^{1/2}\le y\le U,\ |\theta|\le T/U).
\end{equation}

Set $G=\lceil H_rT\rceil$ and $y_j=\lfloor jU/G\rfloor$ for
$0\le j\le G$. For sufficiently large $U$ we have $y_1\ge U^{1/2}$,
and $y_{j+1}-y_j\le\Delta=\lceil U/G\rceil$ for every $j$. Bounding the
maximum over $j$ by the sum and using \eqref{eq:displaced-remainder},
we obtain
\begin{align}
 \frac1{\varphi(r)}\sum_{b\bmod r}^{*}\int_{-T/U}^{T/U}
 \max_{1\le j\le G}|R_{y_j}(b,\theta)|^2\,d\theta
 &\le\sum_{j=1}^G\int_{-T/U}^{T/U}
 \|R_{y_j}(\cdot,\theta)\|_{2,r}^2\,d\theta\notag\\
 &\ll\frac UrH_rT^4\ell^{-3/2}\ll\frac Ur.\label{eq:small-endpoints}
\end{align}
Here $T^4\le\ell$, and \eqref{eq:subpower} with $\nu=1/128$ gives
$H_r\ll r^{1/128}\le\ell^{1/32}$, so that
$H_rT^4\ell^{-3/2}\ll\ell^{-15/32}$.

The interval $I=[-T/U,T/U]$ has length $2T/U\le1/r$. Hence
\eqref{eq:finite-rank-energy}, with $J\le63$, gives
\[
 \frac1{\varphi(r)}\sum_{b\bmod r}^{*}\int_I
 \max_{y\le U}|P_\Eset F_y(b,\theta)|^2\,d\theta\ll\frac Ur.
\]

We now apply Lemma~\ref{lem:coarse-grid} with this interval $I$, with
$P=P_\Eset$ and with the points $y_j$, so that $W=y_G=\lfloor U\rfloor$.
The coefficients $f(n)$ with $n\le U$ are supported on
$n\le\lfloor U\rfloor$, so the lemma applies with $\lfloor U\rfloor$ in
place of $U$, and this does not change the sums $F_y$. Since
$M_f(U;b/r+\theta)=\max_{y\le U}|F_y(b,\theta)|$, the left side of
\eqref{eq:coarse-grid-transfer} is $\EE(U,r,T)$. Since $R_{y_0}=0$,
the first two terms on the right of \eqref{eq:coarse-grid-transfer}
have just been bounded. Since $\sum_{n\le U}|f(n)|^2\le U$ and
$|I|=2T/U$, the third term is
\begin{equation}\label{eq:cell-energy}
 \ll\frac{T(\Delta+r)}{\varphi(r)}\ll\frac Ur+TH_r\ll\frac Ur.
\end{equation}
The last step holds because $r$, $T$ and $H_r$ are at most fixed
powers of $\ell$. This proves \eqref{eq:energy-small} for
$f\in\mathcal B_1$.
\end{proof}

\subsection{A fixed exceptional set for short intervals}\label{sec:short-characters}

We now bound sums of $f(n)\overline{\chi(n)}$ over short intervals,
with a set of exceptional characters that is the same for all the
intervals used in Section~\ref{sec:large-energy}. The first lemma and
its proof follow \cite[Proposition~5.2]{GHS};
Proposition~\ref{prop:GHS-short-input}(ii) allows $|t_j|$ to be as
large as a fixed power of $\log X$.

\begin{lemma}\label{lem:repulsion}
Let $g$ be completely multiplicative with $|g(n)|\le1$, let $K\ge1$ be
fixed and let $J\ge1$ be an integer. Let $X$ be sufficiently large,
$r\log X\le y\le\sqrt X$, and let $\chi_1,\ldots,\chi_J$ be distinct
characters modulo $r$. Put
\[
 \mathcal G_\chi(s;y)
 =\prod_{p>y}\big(1-g(p)\overline{\chi(p)}p^{-s}\big)^{-1}.
\]
Then, for all real $\delta_j\ge0$ and $t_j$ with $|t_j|\le(\log X)^K$,
\begin{equation}\label{eq:repulsion}
 \min_{1\le j\le J}
 \log\big|\mathcal G_{\chi_j}(1+1/\log X+\delta_j+it_j;y)\big|
 \le\frac1{\sqrt J}\log\frac{\log X}{\log y}+O_K(\sqrt J).
\end{equation}
\end{lemma}

\begin{proof}
Write $p(n)$ for the least prime factor of an integer $n>1$, put
$\sigma_0=1+1/\log X$ and
\[
 Z=\sum_{p(n)>y}\frac{\Lambda(n)}{n^{\sigma_0}\log n}
 =\log\frac{\log X}{\log y}+O(1).
\]
The sum is $\log\zeta(\sigma_0)=\log\log X+O(1)$ less the contribution
of the primes $p\le y$, which is $\log\log y+O(1)$ by Mertens' theorem
\cite[Theorem~2.7(d)]{MV2},
since $0\le p^{-k}-p^{-k\sigma_0}\le kp^{-k}(\log p)/\log X$ for
$k\ge1$, $\sum_{p\le y}(\log p)/p\ll\log y$ by \cite[(2.2)]{Robles},
and $y\le\sqrt X$. Since $g$ is completely multiplicative,
$\log\mathcal G_\chi(s;y)$ is the sum of
$\Lambda(n)g(n)\overline{\chi(n)}/(n^s\log n)$ over $p(n)>y$ when
$\Re s>1$. Put $s_j=\sigma_0+\delta_j+it_j$. Taking real parts, using
$|g(n)|\le1$ and applying the Cauchy-Schwarz inequality, we obtain
\[
 \sum_{j=1}^J\log|\mathcal G_{\chi_j}(s_j;y)|
 \le\bigg(Z\sum_{p(n)>y}\frac{\Lambda(n)}{n^{\sigma_0}\log n}
 \bigg|\sum_{j=1}^J\overline{\chi_j(n)}n^{-\delta_j-it_j}\bigg|^2
 \bigg)^{1/2}.
\]
We expand the square. The terms with $j=k$ contribute at most $JZ$ to
the sum over $n$. For $j\ne k$, the term $(j,k)$ contributes
\[
 \log L_y\big(\sigma_0+\delta_j+\delta_k+i(t_j-t_k),
 \overline{\chi_j}\chi_k\big),
\]
whose real part is the logarithm of the modulus. The character
$\overline{\chi_j}\chi_k$ is nonprincipal and
$|t_j-t_k|\le2(\log X)^K\le(\log X)^{K+1}$, so
Proposition~\ref{prop:GHS-short-input}(ii), with $K+1$ in place of $K$,
shows that this real part is at most $O_K(1)$. The sum over $n$ is
real, and is therefore at most $JZ+O_K(J^2)$. Hence
\[
 \sum_{j=1}^J\log|\mathcal G_{\chi_j}(s_j;y)|
 \le\sqrt{Z\big(JZ+O_K(J^2)\big)}\le\sqrt J\,Z+O_K(J^{3/2}).
\]
The minimum does not exceed the mean, and \eqref{eq:repulsion} follows
on dividing by $J$.
\end{proof}

The next lemma is the estimate for short intervals. The exceptional
set depends on $K$, but not on the interval.

\begin{lemma}\label{lem:short-characters}
Let $K\ge1$ be fixed, let $f\in\BA$, let $U$ be sufficiently large in
terms of $A$ and $K$, and let $r\le\ell^4$. There is a set $\Eset$ of at
most $4095$ characters modulo $r$ such that, for every
$\chi\notin\Eset$,
\begin{equation}\label{eq:short-character}
 \bigg|\sum_{x<n\le x+z}f(n)\overline{\chi(n)}\bigg|
 \ll_{A,K}z\ell^{-31/32}
\end{equation}
for all $x$ and $z$ with $U^{1/2}\le x\le U$ and $x/\ell^K\le z\le x$.
\end{lemma}

\begin{proof}
Let $g$ and $h$ be as in Lemma~\ref{lem:companion}. Put $X=4U$ and
$y=r\log X$, so that $y\le\sqrt X$ and $\log\ell\le\log y\ll\log\ell$
for large $U$. For a constant $C_K$, let $\Eset$ be the set of
characters $\chi$ modulo $r$ for which
\[
 \sup_{\substack{\sigma\ge1+1/\log X\\|t|\le(\log X)^{K+3}}}
 |\mathcal G_\chi(\sigma+it;y)|
 >C_K\Big(\frac{\log X}{\log y}\Big)^{1/64}.
\]
If $\Eset$ contained $4096$ distinct characters $\chi_j$, there would
be for each of them a point $\sigma_j+it_j$ in the stated range at which
$|\mathcal G_{\chi_j}(\sigma_j+it_j;y)|$ exceeds the right side. This
contradicts Lemma~\ref{lem:repulsion}, applied with $J=4096$ and with
$K+3$ in place of $K$, once $\log C_K$ is larger than the term
$O_K(\sqrt J)$ in \eqref{eq:repulsion}. With this choice of $C_K$, the
set $\Eset$ has at most $4095$ elements. It depends only on $g$, $r$,
$U$ and $K$.

Write $G_\chi(s)=\sum_{n\ge1}g(n)\overline{\chi(n)}n^{-s}$. This is the
product of $\mathcal G_\chi(s;y)$ and the Euler factors at the primes
$p\le y$, and for $\Re s>1$ these factors have a product of modulus at
most $\prod_{p\le y}(1-p^{-1})^{-1}\ll\log y\ll\log\ell$. If
$U^{1/4}\le x\le2U$, then $1+1/\log x\ge1+1/\log X$. Hence, for
$\chi\notin\Eset$, $U^{1/4}\le x\le2U$ and $|t|\le(\log X)^{K+3}$,
\begin{equation}\label{eq:G-nonexceptional}
 |G_\chi(1+1/\log x+it)|\ll_K\ell^{1/64}(\log\ell)^{63/64}.
\end{equation}

Now let $\chi\notin\Eset$, $U^{1/4}\le x\le2U$ and $x/\ell^K\le z<x$.
We apply Proposition~\ref{prop:GHS-short-input}(i) to
$u=g\overline\chi\in\mathcal C(1)$, with $\delta=\log(x/z)/\log x$. For large $U$ we
have $0<\delta<1/4$, because $\log(x/z)\le K\log\ell$ and
$\log x\ge\ell/4$, and
$(x/z)\log x\le\ell^K\log(2U)\le(\log X)^{K+3}$. By \eqref{eq:short-M}
and \eqref{eq:G-nonexceptional}, we have $e^{-M}\le b_U$, where
$b_U=C_K'\ell^{-63/64}(\log\ell)^{63/64}$ for a suitable constant
$C_K'\ge1$. For sufficiently large $U$ we have $b_U<1$, so that
$M\ge\log(1/b_U)>0$. Since $(1+m)e^{-m}$ is decreasing for $m\ge0$,
\[
 (1+M)e^{-M}\le\big(1+\log(1/b_U)\big)b_U
 \ll_K\ell^{-63/64}(\log\ell)^2.
\]
The second term on the right of \eqref{eq:short-halasz} is
$\ll_K(\log\ell)/\ell$. Therefore
\begin{equation}\label{eq:g-short}
 \bigg|\sum_{x<n\le x+z}g(n)\overline{\chi(n)}\bigg|
 \ll_Kz\ell^{-31/32}\qquad(\chi\notin\Eset)
\end{equation}
for $U^{1/4}\le x\le2U$ and $x/\ell^K\le z<x$. If $z=x$ and
$U^{1/4}\le x\le U$, we divide $(x,2x]$ into the intervals $(x,3x/2]$
and $(3x/2,2x]$. Their left endpoints are in $[U^{1/4},2U]$, and their
length $x/2$ is smaller than the left endpoint and at least $\ell^{-K}$
times it. Thus \eqref{eq:g-short} also holds for $z=x$ when
$U^{1/4}\le x\le U$.

Finally let $U^{1/2}\le x\le U$ and $x/\ell^K\le z\le x$. Since
$f=g*h$,
\[
 \sum_{x<n\le x+z}f(n)\overline{\chi(n)}
 =\sum_{d\le x+z}h(d)\overline{\chi(d)}
 \sum_{x/d<m\le(x+z)/d}g(m)\overline{\chi(m)}.
\]
For $d\le U^{1/8}$ we have $U^{3/8}\le x/d\le U$, and the ratio of the
length $z/d$ to the left endpoint $x/d$ is $z/x$. By \eqref{eq:g-short}
and \eqref{eq:h-norm} with $\sigma=1$, these $d$ contribute
$\ll_{A,K}z\ell^{-31/32}$. For $d>U^{1/8}$ the inner sum has at most
$z/d+1$ terms, and \eqref{eq:h-tail} gives
\[
 \sum_{U^{1/8}<d\le2x}|h(d)|\Big(\frac zd+1\Big)
 \ll_AzU^{-1/32}+(2x)^{3/4}.
\]
Since $z\ge x/\ell^K$ and $x\ge U^{1/2}$, both terms are
$\ll_Kz\ell^{-31/32}$. This proves \eqref{eq:short-character}.
\end{proof}

In Section~\ref{sec:large-energy} the lemma is applied to the inner
sums of \eqref{eq:smooth-expansion}, over intervals of arbitrary
length, in the following form.

\begin{corollary}\label{cor:short-dilated}
Let $f\in\BA$, let $U$ be sufficiently large in terms of $A$, and let
$r\le\ell^4$. Let $\Eset$ be the set of
Lemma~\ref{lem:short-characters} with $K=16$, and write
$\epsU=\ell^{-31/32}$. Let $U^{3/4}\le v\le v'\le U$ and
$1\le s\le U^{1/8}$. Then, for every $\chi\notin\Eset$,
\begin{equation}\label{eq:short-dilated}
 \bigg|\sum_{v/s<m\le v'/s}f(m)\overline{\chi(m)}\bigg|
 \ll_A\frac{v'-v}s\epsU+\frac Us\ell^{-16}.
\end{equation}
\end{corollary}

\begin{proof}
The sum is over an interval with left endpoint $v/s\ge U^{5/8}$,
length $(v'-v)/s$ and right endpoint at most $U/s$. As long as the length
of the remaining interval exceeds its left endpoint $x$, we remove the
interval $(x,2x]$, to which \eqref{eq:short-character} applies with
$z=x$. This leaves an interval $(x,x+z]$ with $z\le x\le U/s$. If
$z\ge x/\ell^{16}$, then \eqref{eq:short-character} applies to it as
well. Otherwise it contains at most $z+1\le(U/s)\ell^{-16}+1$ integers,
and since $|f(m)|\le A$ its contribution is $\ll_A(U/s)\ell^{-16}$. All
the left endpoints lie between $U^{1/2}$ and $U$, and the lengths of
the intervals to which \eqref{eq:short-character} has been applied
have sum at most $(v'-v)/s$. This proves \eqref{eq:short-dilated}.
\end{proof}

\subsection{Wider arcs}\label{sec:large-energy}

The second estimate allows $T$ up to $\ell^8$, for a smaller range of
$r$.

\begin{proposition}\label{prop:energy-large}
There is an absolute constant $U_0$ such that, if $f\in\BA$,
$U\ge U_0$, $r\le(\log U)^{15/8}$ and $1\le T\le(\log U)^8$, then
\begin{equation}\label{eq:energy-large}
 \EE(U,r,T)\ll A^2\frac Ur,
\end{equation}
with an absolute implied constant.
\end{proposition}

\begin{proof}
By \eqref{eq:energy-A-transfer}, we may assume $f\in\mathcal B_1$. All
constants and all lower bounds for $U$ in this proof are then absolute.
Put $I=[-T/U,T/U]$ and $h_0=U/(2T)$, and let $\Eset$ be the set of
Corollary~\ref{cor:short-dilated}.

We first remove the terms with $n\le U^{3/4}$. Let $c_n=f(n)$ for
$U^{3/4}<n\le U$ and $c_n=0$ otherwise. Then
$M_f(U;\alpha)\le M_{f-c}(U;\alpha)+M_c(U;\alpha)$ for every real
$\alpha$, and the coefficients $f(n)-c_n$ with $n\le U$ are supported
on $n\le U^{3/4}$. Since $|I|<1$, the estimate \eqref{eq:CH} gives
$\int_IM_{f-c}(U;b/r+\theta)^2\,d\theta\ll U^{3/4}$ for every $b$, and
this is $\ll U/r$. In the rest of the proof $F_y$ is formed with the
coefficients $c_n$, and $R_y=(1-P_\Eset)F_y$. It remains to show that
\begin{equation}\label{eq:large-goal}
 \big\|\max_{y\le U}|F_y|\big\|_{I,r}^2\ll\frac Ur,
\end{equation}
in the notation of Lemma~\ref{lem:coarse-grid}.

Let $a_n=(1-P_\Eset)\big[c_n\e(nb/r)\big]$, a function of the reduced
residues $b$ modulo $r$, so that
$R_y(\cdot,\theta)=\sum_{n\le y}a_n\e(n\theta)$. We claim that
\begin{equation}\label{eq:window-vector}
 \bigg\|\sum_{v<n\le v'}a_n\bigg\|_{2,r}\ll h_0\epsU H_rV_r
 \qquad(U^{3/4}\le v\le v'\le U,\ v'-v\le h_0).
\end{equation}
By Parseval's identity for \eqref{eq:char-expansion} and by
\eqref{eq:smooth-expansion}, the square of the left side is
\[
 \sum_{\chi\notin\Eset}\bigg|\sum_{s\in\Sr}f(s)\Gamma_\chi(s;r)
 \sum_{v/s<m\le v'/s}f(m)\overline{\chi(m)}\bigg|^2.
\]
For $s\le U^{1/8}$ the inner sum is
$\ll(h_0\epsU+U\ell^{-16})/s$ by \eqref{eq:short-dilated}. Since there
are at most $r$ characters, \eqref{eq:W} shows that the terms with
$s\le U^{1/8}$ contribute to the norm at most
\[
 C\sqrt r\,W_r\big(h_0\epsU+U\ell^{-16}\big)
 \le CH_rV_r\big(h_0\epsU+U\ell^{-16}\big).
\]
For $s>U^{1/8}$ the inner sum has at most $h_0/s+1$ terms and
$|\Gamma_\chi(s;r)|\le1$, so by \eqref{eq:smooth-tail} these $s$
contribute at most $r^{5/2}(h_0U^{-1/16}+U^{1/2})$. Now $T\le\ell^8$ and
$r\le\ell^{15/8}$ give $h_0\epsU\ge U\ell^{-9}/2$ and
$r^{5/2}\le\ell^5$. Hence $U\ell^{-16}$, $r^{5/2}h_0U^{-1/16}$ and
$r^{5/2}U^{1/2}$ are all $\ll h_0\epsU$. As $H_rV_r\ge1$, this proves
\eqref{eq:window-vector}.

Next let $U^{3/4}\le v\le v'\le U$, with no condition on $v'-v$, and
apply Lemma~\ref{lem:Gallagher} with $\delta=T/U$ to the functions
$a_n$ with $v<n\le v'$. For every real $t$, the integers $n$ with
$t<n\le t+h_0$ and $v<n\le v'$ are those of an interval to which
\eqref{eq:window-vector} applies, and there are none unless
$v-h_0<t<v'$. Therefore
\begin{equation}\label{eq:interval-energy}
 \int_I\bigg\|\sum_{v<n\le v'}a_n\e(n\theta)\bigg\|_{2,r}^2\,d\theta
 \ll(v'-v+h_0)\epsU^2H_r^2V_r^2.
\end{equation}

Let $W$ be the least power of $2$ with $W\ge U$, and let $G$ be the
power of $2$ with $H_rT\le G<2H_rT$. For sufficiently large $U$ we have
$G\le W$, so that $\Delta=W/G$ is an integer; put $y_j=j\Delta$ for
$0\le j\le G$. For $0\le k\le\log_2G$, let $\mathcal D_k$ be the set of
the $2^k$ intervals $((i-1)W/2^k,iW/2^k]$ with $1\le i\le2^k$. By the
binary expansion of $j/G$, each interval $(0,y_j]$ with $1\le j\le G$
is a disjoint union of intervals taken from the sets $\mathcal D_k$, at
most one from each. Hence, for all $b$ and $\theta$,
\[
 \max_{0\le j\le G}|R_{y_j}(b,\theta)|
 \le\sum_{k=0}^{\log_2G}\bigg(\sum_{D\in\mathcal D_k}
 \bigg|\sum_{n\in D}a_n(b)\e(n\theta)\bigg|^2\bigg)^{1/2},
\]
and Minkowski's inequality for the norm $\|\cdot\|_{I,r}$ gives
\begin{equation}\label{eq:dyadic-grid-maximal}
 \big\|\max_{0\le j\le G}|R_{y_j}|\big\|_{I,r}
 \le\sum_{k=0}^{\log_2G}\bigg(\sum_{D\in\mathcal D_k}\int_I
 \bigg\|\sum_{n\in D}a_n\e(n\theta)\bigg\|_{2,r}^2\,d\theta\bigg)^{1/2}.
\end{equation}
Since $a_n=0$ unless $U^{3/4}<n\le U$, the estimate
\eqref{eq:interval-energy} applies to the intersection of each $D$ with
$(U^{3/4},U]$. For fixed $k$ these intersections have total length at
most $U$, and there are at most $G$ of them. As $Gh_0<H_rU$, the sum
over $D\in\mathcal D_k$ in \eqref{eq:dyadic-grid-maximal} is
\[
 \ll(U+Gh_0)\epsU^2H_r^2V_r^2\ll U\epsU^2H_r^3V_r^2.
\]
The number of $k$ is $1+\log_2G\ll\log\ell$. Hence
\begin{equation}\label{eq:large-grid}
 \big\|\max_{0\le j\le G}|R_{y_j}|\big\|_{I,r}^2
 \ll U\epsU^2H_r^3V_r^2(\log\ell)^2\ll\frac Ur.
\end{equation}
For the last step, \eqref{eq:subpower} with $\nu=1/128$ and
$r\le\ell^{15/8}$ give
\[
 rH_r^3V_r^2\epsU^2(\log\ell)^2
 \ll r^{129/128}\ell^{-31/16}(\log\ell)^2
 \le\ell^{-49/1024}(\log\ell)^2.
\]

Since $|I|=2T/U\le1/r$ and $c_n=0$ for $n>U$, the estimate
\eqref{eq:finite-rank-energy}, with $J\le4095$, gives
$\|\max_{y\le W}|P_\Eset F_y|\|_{I,r}^2\ll U/r$. We now apply
Lemma~\ref{lem:coarse-grid} to the coefficients $c_n$, with $P=P_\Eset$
and the points $y_j$. The first two terms on the right of
\eqref{eq:coarse-grid-transfer} have just been bounded. Since
$\sum_n|c_n|^2\le U$ and $\Delta<2U/(H_rT)$, the third term is
\[
 \ll\frac{T(\Delta+r)}{\varphi(r)}\ll\frac Ur+TH_r\ll\frac Ur.
\]
This proves \eqref{eq:large-goal}, and with it \eqref{eq:energy-large}
for $f\in\mathcal B_1$.
\end{proof}

\begin{remark}\label{rem:mechanism}
Proposition~\ref{prop:sampling} and Lemmas~\ref{lem:Gallagher}
and~\ref{lem:coarse-grid} hold for arbitrary coefficients. In the
proofs of Propositions~\ref{prop:energy-small}
and~\ref{prop:energy-large}, multiplicativity is used through
Lemmas~\ref{lem:companion} and~\ref{lem:linear-A}, the character sum
estimates
\eqref{eq:GHS-long} and \eqref{eq:short-character}, and the expansion
\eqref{eq:smooth-expansion}.
\end{remark}

\section{Proofs of the main results}\label{sec:proofs}

\subsection{Proof of Theorem~\ref{thm:aux}}\label{sec:prime-blocks}

Recall the sums $T(P)$ from \eqref{eq:TP-intro}. The first step is the
bound for $T(P)$ that results from the estimates of
Section~\ref{sec:energy}.

\begin{lemma}\label{lem:new-block}
Assume the hypotheses of Proposition~\ref{prop:old-blocks}(ii), with
$f\in\BA$, and let $P$ be dyadic with $Y\le P<N$ and
$U=\lfloor N/P\rfloor$. If the hypotheses of
Proposition~\ref{prop:energy-small} or of
Proposition~\ref{prop:energy-large} hold with $T=4B$, then
\begin{equation}\label{eq:new-block}
 |T(P)|\ll_A\frac N{\sqrt{rB}}.
\end{equation}
\end{lemma}

\begin{proof}
Let $Q$ be the left side of \eqref{eq:positive-transfer} with
$c_n=f(n)$, so that $d_n=nf(n)$. For every integer $k\le U$ and every
real $\xi$, partial summation gives
\[
 \sum_{n\le k}nf(n)\e(n\xi)=k\sum_{n\le k}f(n)\e(n\xi)
 -\sum_{m<k}\sum_{n\le m}f(n)\e(n\xi),
\]
and therefore
\begin{equation}\label{eq:weighted-max}
 M_d(U;\xi)\le2UM_f(U;\xi).
\end{equation}
Since $h|\beta|U\le1$, it follows from \eqref{eq:positive-transfer}
that
\begin{equation}\label{eq:Q-transfer}
 Q\ll\frac1{\varphi(r)}\sum_{b\bmod r}^{*}\int_P^{2P}
 M_f(U;ab/r+\beta t)^2\,dt.
\end{equation}
We show that
\begin{equation}\label{eq:Q-bound}
 Q\ll_A\frac{NU}{rB}.
\end{equation}

Suppose first that $B>1$, so that $|\beta|=B/N$. For $P\le t\le2P$ we
have $|\beta t|\le2P|\beta|\le2B/U$, because $PU\le N$, and
multiplication by $a$ permutes the reduced residues modulo $r$. The
substitution $\theta=\beta t$ in \eqref{eq:Q-transfer} therefore gives
\[
 Q\ll\frac1{|\beta|}\EE(U,r,2B)\le\frac NB\EE(U,r,4B),
\]
and \eqref{eq:Q-bound} follows from \eqref{eq:energy-small} or
\eqref{eq:energy-large}.

Now suppose that $B=1$, so that $N|\beta|\le1$; here $\beta$ may be
$0$. Put $I=[-4/U,4/U]$. For an integer $k\le U$ and a reduced residue
$b$ modulo $r$, let
$F(\theta)=\sum_{n\le k}f(n)\e(n(b/r+\theta))$. By
\cite[Lemma~4.4]{Robles}, applied on the interval $I$ of length $8/U$,
\[
 \sup_{\theta\in I}|F(\theta)|^2\le\frac U4\int_I|F(\theta)|^2\,d\theta
 +\frac{16}U\int_I|F'(\theta)|^2\,d\theta.
\]
Here $|F(\theta)|\le M_f(U;b/r+\theta)$ and, by
\eqref{eq:weighted-max},
$|F'(\theta)|\le2\pi M_d(U;b/r+\theta)\le4\pi UM_f(U;b/r+\theta)$.
Taking the maximum over $k$ and the mean over $b$, we obtain, for every
$\theta_0\in I$,
\[
 \frac1{\varphi(r)}\sum_{b\bmod r}^{*}M_f(U;b/r+\theta_0)^2
 \ll U\,\EE(U,r,4)\ll_A\frac{U^2}r,
\]
by \eqref{eq:energy-small} or \eqref{eq:energy-large} with $T=4$. For
$P\le t\le2P$ we have $|\beta t|\le2P/N\le2/U$, so that $\beta t\in I$.
Using the last bound in \eqref{eq:Q-transfer}, again after replacing
$ab$ by $b$, we get $Q\ll_APU^2/r\le NU/r$, which is
\eqref{eq:Q-bound}.

Finally, for $P<p\le2P$ we have $N/p<N/P$, hence
$|S_f(N/p,p\alpha)|\le M_f(U;\alpha p)$. By the Cauchy-Schwarz
inequality, the bound $\sum_{P<p\le2P}\log p\ll P$ of
\cite[(2.2)]{Robles} and \eqref{eq:Q-bound},
\[
 |T(P)|^2\le\bigg(\sum_{P<p\le2P}|f(p)|^2\log p\bigg)Q
 \ll_AP\frac{NU}{rB}\le\frac{N^2}{rB}.
\]
This proves \eqref{eq:new-block}.
\end{proof}

\begin{proof}[Proof of Theorem~\ref{thm:aux}]
By \eqref{eq:linear-A} with $X=N$ and $c_n=\e(n\alpha)$, it suffices to
prove \eqref{eq:aux-main} for $f\in\mathcal B_1$, with an absolute
implied constant. Let $f\in\mathcal B_1$. For bounded $N$ the assertion
follows from $|S_f(N,\alpha)|\le N$, and for $r\ge L^3$ from
Proposition~\ref{prop:old-blocks}(i). We may therefore assume that $N$
is sufficiently large and that $r<L^3$, and we use
\eqref{eq:reduction}.

Let $U_0$ be an absolute constant for which
Propositions~\ref{prop:energy-small} and~\ref{prop:energy-large} both
hold, and put $w_0=\log U_0$. Define
\begin{equation}\label{eq:cutoffs}
 w=\begin{cases}
 \max\{w_0,r^{1/4},(4B)^4\}&\text{if }B\le r^{1/8},\\
 \max\{w_0,r^{8/15},(4B)^{1/8}\}&\text{if }B>r^{1/8},
 \end{cases}
 \qquad U_*=e^w.
\end{equation}
Let $P$ be dyadic with $Y\le P<N$, and $U=\lfloor N/P\rfloor$. If
$U\ge U_*$, then $U\ge U_0$ and $\log U\ge w$. In the first case of
\eqref{eq:cutoffs} this gives $r\le(\log U)^4$ and $4B\le(\log U)^{1/4}$, the hypotheses
of Proposition~\ref{prop:energy-small} with $T=4B$. In the second case
it gives $r\le(\log U)^{15/8}$ and $4B\le(\log U)^8$, the hypotheses of
Proposition~\ref{prop:energy-large}. Hence Lemma~\ref{lem:new-block}
gives $T(P)\ll N/\sqrt{rB}$ whenever $U\ge U_*$. We divide the sum in
\eqref{eq:reduction} into three parts.

The $P$ with $P\le N/r$ and $U\ge U_*$ number $O(L)$, and by the bound
just obtained their contribution is $O(NL/\sqrt{rB})$.

For the $P$ with $P\le N/r$ and $U<U_*$ we use \eqref{eq:old-block}.
For these $P$ we have $N/P<U_*+1$, so their number is $O(1+w)$. The
second terms of \eqref{eq:old-block} form a geometric sum:
\[
 \sum_{\substack{P\le N/r\\P\textup{ dyadic}}}
 \sqrt{\frac{rNP}{\varphi(r)}}\ll\frac N{\sqrt{\varphi(r)}}\le N.
\]
The first terms contribute
\begin{equation}\label{eq:short-first}
 \ll\frac{(1+w)N}{\sqrt{\varphi(r)B}}
 =\frac{(1+w)\sqrt{H_r}}{\sqrt{rB}}N\ll N.
\end{equation}
To prove the last step, suppose first that $B\le r^{1/8}$. Then
$1+w\ll r^{1/4}+B^4$ and
\[
 \frac{1+w}{\sqrt{rB}}\ll r^{-1/4}B^{-1/2}+r^{-1/2}B^{7/2}
 \le2r^{-1/16}.
\]
If $B>r^{1/8}$, then $1+w\ll r^{8/15}+B^{1/8}$ and
\[
 \frac{1+w}{\sqrt{rB}}\ll r^{1/30}B^{-1/2}+r^{-1/2}B^{-3/8}
 \le2r^{-7/240}.
\]
In both cases \eqref{eq:short-first} follows from \eqref{eq:subpower}.

For the $P$ with $P>N/r$ we have $U<r$, and we use
\eqref{eq:old-short}. For each $j\ge0$ there is at most one of these $P$
with $2^j\le N/P<2^{j+1}$, and then $2^j\le U<2^{j+1}$. These $P$
therefore contribute
\[
 \ll N\sum_{j\ge0}2^{-j/2}\sqrt{\log\log(3\cdot2^{j+1})}\ll N.
\]

Adding the three parts, we obtain from \eqref{eq:reduction} that
$L\,S_f(N,\alpha)\ll NL/\sqrt{rB}+N$, and \eqref{eq:aux-main} for
$f\in\mathcal B_1$ follows on dividing by $L$.
\end{proof}

\subsection{Proofs of Theorems~\ref{thm:global} and~\ref{thm:refined}}\label{sec:deductions}

Both theorems are deduced from Theorem~\ref{thm:aux} by comparing the
given fraction $a/q$ with a fraction supplied by Dirichlet's theorem,
which we use in the form of \cite[Lemma~2.2]{Robles}: for every real
$Q\ge1$ and every real $\alpha$ there are integers $a_1$ and $r$ with
\begin{equation}\label{eq:dirichlet}
 (a_1,r)=1,\qquad 1\le r\le Q,\qquad
 \bigg|\alpha-\frac{a_1}r\bigg|\le\frac1{rQ}.
\end{equation}
In both proofs we may assume that $N$ is so large that
$Q_0=N/L^3\ge1$, since for bounded $N$ the assertions follow from
$|S_f(N,\alpha)|\le AN$. We apply \eqref{eq:dirichlet} with $Q=Q_0$ and
write
\[
 \beta_1=\alpha-\frac{a_1}r,\qquad B_1=\max\{1,N|\beta_1|\}.
\]
Then \eqref{eq:aux-hyp} holds with $a_1$ and $\beta_1$ in place of $a$
and $\beta$, and Theorem~\ref{thm:aux} gives
\begin{equation}\label{eq:aux-applied}
 |S_f(N,\alpha)|\ll A\bigg(\frac NL+\frac N{\sqrt{rB_1}}\bigg).
\end{equation}

\begin{proof}[Proof of Theorem~\ref{thm:global}]
Let $(a,q)=1$ and $|\alpha-a/q|\le q^{-2}$. If $r\ge q/2$, then
$rB_1\ge q/2$, and the right side of \eqref{eq:aux-applied} is
$\ll A(N/L+N/\sqrt q)$. If $r<q/2$, the reduced fractions $a_1/r$ and
$a/q$ are distinct, so that $|a_1/r-a/q|\ge1/(rq)$ and
\[
 |\beta_1|\ge\frac1{rq}-\frac1{q^2}\ge\frac1{2rq}.
\]
Hence $rB_1\ge N/(2q)$, and the right side of \eqref{eq:aux-applied} is
$\ll A(N/L+\sqrt{Nq})$. This proves \eqref{eq:three-term}. If
$R\le q\le N/R$, then $N/\sqrt q$ and $\sqrt{Nq}$ are both at most
$N/\sqrt R$, which gives \eqref{eq:global}. If $q\le N/L^2$, then
$\sqrt{Nq}\le N/L$, which gives \eqref{eq:folklore}.
\end{proof}

\begin{proof}[Proof of Theorem~\ref{thm:refined}]
Here $\alpha=a/q+\beta$. If $a_1/r=a/q$, then $r=q$, $\beta_1=\beta$
and $rB_1=D$, and the right side of \eqref{eq:aux-applied} is
$A(N/L+N/\sqrt D)$. Otherwise, since $B\ge1$ and $B_1\ge1$,
\[
 \frac1{qr}\le|\beta|+|\beta_1|\le\frac{B+B_1}N\le\frac{2BB_1}N.
\]
Thus $D\,rB_1\ge N/2$, and the right side of \eqref{eq:aux-applied} is
$\ll A(N/L+\sqrt{ND})$. This proves \eqref{eq:displacement-global}.

If $D\le N/L^2$, then $\sqrt{ND}\le N/L$, and \eqref{eq:refined}
follows. Since $D=\max\{q,qN|\beta|\}$, the condition
\eqref{eq:wider-local} is equivalent to $q\le N/L^2$ and
$|\beta|\le1/(qL^2)$. If $q\le Q_2$ and $|\beta|\le1/(qQ_2)$, then
$qN|\beta|\le L^2$, so that $D\le\max\{q,L^2\}\le N/L^2$ as soon as
$N\ge L^4$; the remaining $N$ are bounded. Finally,
$\max\{1,x\}\le1+x\le2\max\{1,x\}$ for $x\ge0$.
\end{proof}

The arcs in Theorem~\ref{thm:refined} can be defined with any fixed
power of $L$; those of Theorem~\ref{thm:aux} correspond to $C=3$.

\begin{corollary}\label{cor:fixed-log}
Let $C\ge2$ be fixed. If $f\in\BA$, $(a,q)=1$, $q\le N/L^C$ and
$|\beta|\le L^C/(qN)$, then
\[
 |S_f(N,a/q+\beta)|\ll_CA\bigg(\frac NL
 +\frac N{\sqrt{q\max\{1,N|\beta|\}}}\bigg).
\]
\end{corollary}

\begin{proof}
With $D$ as in Theorem~\ref{thm:refined}, we have
$D\le\max\{q,L^C\}$ and $q\le N/L^C\le N/L^2$. If $N$ is sufficiently
large in terms of $C$, then also $L^C\le N/L^2$, and
\eqref{eq:refined} applies. For the remaining $N$ the assertion follows
from $|S_f(N,a/q+\beta)|\le AN$.
\end{proof}

\subsection{Logarithmic weights}\label{sec:log-weights}

By partial summation, Theorem~\ref{thm:global} gives the following
bound.

\begin{corollary}\label{cor:log-weights}
Let $f\in\BA$, $(a,q)=1$, $1\le q\le N$ and $|\alpha-a/q|\le q^{-2}$.
Then
\begin{equation}\label{eq:log-weights}
 \bigg|\sum_{n\le N}\frac{f(n)}n\e(n\alpha)\bigg|
 \ll A\bigg(\log(2q)+\frac L{\sqrt q}\bigg).
\end{equation}
\end{corollary}

\begin{proof}
Since $|f(n)|\le A$, the sum of $|f(n)|/n$ over $n\le x$ is at most
$A(1+\log x)$ for every $x\ge1$. This gives \eqref{eq:log-weights} when
$q=1$, and also when $q^2\ge N$, since then $L\le2\log q$. Let
$2\le q<\sqrt N$. The terms with $n\le q^2$ contribute
$\ll A\log(2q)$. For $q^2\le t\le N$ we have $2\le q\le t/q$, and
\eqref{eq:global}, with $t$ in place of $N$ and with $R=q$, gives
\[
 |S_f(t,\alpha)|\ll A\bigg(\frac t{\log t}+\frac t{\sqrt q}\bigg).
\]
By partial summation,
\begin{align*}
 \bigg|\sum_{q^2<n\le N}\frac{f(n)}n\e(n\alpha)\bigg|
 &\le\frac{|S_f(N,\alpha)|}N+\frac{|S_f(q^2,\alpha)|}{q^2}
 +\int_{q^2}^N\frac{|S_f(t,\alpha)|}{t^2}\,dt\\
 &\ll A\bigg(1+\log(2L)+\frac L{\sqrt q}\bigg).
\end{align*}
If $q\ge L$, then $\log(2L)\le\log(2q)$. If $q<L$, then
$L/\sqrt q>\sqrt L\gg\log(2L)$. This proves \eqref{eq:log-weights}.
\end{proof}

For $A=1$, the estimate \eqref{eq:log-weights} is the bound conjectured
in \cite[(1.5)]{GL}. It is stated there for $q\le N$ and
$|\alpha-a/q|\le1/(qN)$, which implies $|\alpha-a/q|\le q^{-2}$. The
leading constant proposed in \cite[(1.6)]{GL} is a separate question;
see Section~\ref{sec:constant}.

\section{Sharpness}\label{sec:sharpness}

This section contains the examples behind the assertions made after
Theorem~\ref{thm:refined}. Except in Sections~\ref{sec:A-sharpness}
and~\ref{sec:constant}, the functions are completely multiplicative and
belong to $\mathcal B_1$. The basic examples are due to Montgomery and
Vaughan \cite[Section~7]{MV}; see also \cite[Section~1]{MS}.
Sections~\ref{sec:density} and~\ref{sec:constant} contain further
examples, which concern refinements proposed in \cite{Bachman} and in
\cite{GL}.

\subsection{The $N/\log N$ term}\label{sec:NlogN-sharpness}

For every real $\alpha$ and every large $N$ there is a completely
multiplicative $f\in\mathcal B_1$ with
\[
 |S_f(N,\alpha)|\ge\pi(N)-\pi(N/2)\gg\frac NL.
\]
The values $f(p)$ at the primes $N/2<p\le N$ are chosen so that the
terms $f(p)\e(p\alpha)$ all have the argument of the sum over the
remaining $n\le N$; see \cite[Section~7]{MV} and
\cite[Section~1, footnote~1]{dBG}. Taking $\alpha=1/q$ with
$q=\lfloor\sqrt N\rfloor$ and $R=q$ in \eqref{eq:global}, we see that
the term $N/L$ cannot be decreased even when $N/\sqrt R=o(N/L)$.

\subsection{The $q^{-1/2}$ and $R^{-1/2}$ terms}\label{sec:q-sharpness}

Let $\chi$ be a primitive character modulo $q$, let $(a,q)=1$, and let
$N$ be a multiple of $q$. Summing over the residue classes modulo $q$
and using \eqref{eq:induced-Gauss}, we obtain
\begin{equation}\label{eq:character-obstruction}
 S_\chi(N,a/q)=\frac Nq\,\overline{\chi(a)}\tau(\chi),
 \qquad|S_\chi(N,a/q)|=\frac N{\sqrt q}.
\end{equation}
Now let $q\ge5$ be prime and let $N$ be a multiple of $q$ with
$L\ge q$. Then $R=q$ satisfies $R\le q\le N/R$, and moreover
$q\le N/L^2$ and $N/L\le N/q$. Hence, as $q\to\infty$, neither
$N/\sqrt R$ in \eqref{eq:global} nor $N/\sqrt q$ in
\eqref{eq:folklore} can be replaced by a function of smaller order.

\subsection{Dependence on the coefficient bound}\label{sec:A-sharpness}

For $A\ge2$, let $f_A(n)=1$ for odd $n$ and $f_A(n)=A$ for even $n$.
Of two coprime integers at most one is even, so $f_A$ is multiplicative,
and $f_A\in\BA$. For even $N\ge4$,
\begin{equation}\label{eq:A-sharpness}
 S_{f_A}(N,1/2)=\frac{(A-1)N}2\ge\frac{AN}4.
\end{equation}
Since $q=R=2$ is admissible in Theorem~\ref{thm:global}, the factor $A$
in \eqref{eq:global} cannot be replaced by a function of smaller order.

The factor $A^2$ in Propositions~\ref{prop:energy-small}
and~\ref{prop:energy-large} is also necessary. Let $U$ be even and
$|\theta|\le1/(12U)$. Then $\cos(2\pi n\theta)\ge1/2$ for $n\le U$, and
therefore
\[
 \bigg|\sum_{n\le U}f_A(n)\e(n\theta)\bigg|
 \ge\frac12\sum_{n\le U}f_A(n)\ge\frac{AU}4.
\]
The left side is at most $M_{f_A}(U;\theta)$. Integrating its square
over $|\theta|\le1/(12U)$ gives, for $r=T=1$,
\begin{equation}\label{eq:A-energy-sharpness}
 \mathcal E_{f_A}(U,1,1)\ge\frac{A^2U}{96}.
\end{equation}

\subsection{The logarithmic terms}\label{sec:log-sharpness}

Both terms of \eqref{eq:log-weights} are necessary, already for
completely multiplicative functions in $\mathcal B_1$; compare
\cite[Sections~1.7 and~1.8]{GL}. For the first term, take $f=1$ and
$\alpha=1/q$ with $q\ge2$. Since $|\e(n/q)-1|\le2\pi n/q$,
\[
 \sum_{n\le q}\frac{\e(n/q)}n=\sum_{n\le q}\frac1n+O(1)=\log q+O(1).
\]
The sums $\sum_{n\le t}\e(n/q)$ are $O(q)$, so by partial summation the
terms with $q<n\le N$ contribute $O(1)$, uniformly for $N\ge q$. Hence
\begin{equation}\label{eq:logq-sharpness}
 \sum_{n\le N}\frac{\e(n/q)}n=\log q+O(1)\qquad(N\ge q\ge2),
\end{equation}
with an absolute implied constant. For $N=q^2$ we have
$L/\sqrt q=o(\log q)$ as $q\to\infty$.

For the second term, let $\chi$ be a primitive character modulo $q$ and
$(a,q)=1$. The sequence $\chi(n)\e(an/q)$ has period $q$, and by
\eqref{eq:character-obstruction} its mean over a period is
$\mu=\overline{\chi(a)}\tau(\chi)/q$, with $|\mu|=q^{-1/2}$. Thus
$S_\chi(t,a/q)=\mu t+O(q)$ uniformly for real $t\ge0$, and partial
summation on $(q,N]$ gives
\[
 \sum_{q<n\le N}\frac{\chi(n)}n\e(an/q)=\mu\log\frac Nq+O(1)
 \qquad(N\ge q).
\]
The terms with $n\le q$ contribute at most $1+\log q$. It follows that
\begin{equation}\label{eq:logN-sharpness}
 \sum_{n\le N}\frac{\chi(n)}n\e(an/q)=\mu L+O(\log(2q)).
\end{equation}
If $q$ is prime and $L\ge q$, the error term is $o(L/\sqrt q)$ as
$q\to\infty$, so the left side has modulus $\sim L/\sqrt q$.

\subsection{An archimedean example}\label{sec:archimedean}

The next two examples use the functions $n^{-i\pi B}$, as in
\cite[Proposition~8.1]{Robles}.

\begin{proposition}\label{prop:stationary}
Let $N\ge3$ be an integer, $B\ge2$ and $f_B(n)=n^{-i\pi B}$. There is
a real number $\vartheta=\vartheta(N,B)$ such that
\begin{equation}\label{eq:stationary}
 S_{f_B}(N,B/N)=\frac N{\sqrt{2B}}\,e^{i\vartheta}
 +O\Big(\frac NB+BL+1\Big).
\end{equation}
In particular, if $B\to\infty$ and $B^{3/2}L=o(N)$, then
$|S_{f_B}(N,B/N)|\sim N/\sqrt{2B}$.
\end{proposition}

\begin{proof}
For $B=L$ this is \cite[(8.2)]{Robles}, with
$\vartheta=\pi B(1+\log2)-\pi B\log N+\pi/4$. The proof given there
applies to every $B\ge2$. It compares $S_{f_B}(N,B/N)$ with the
integral of $\exp(i(2\pi Bx/N-\pi B\log x))$ over $1\le x\le N$, with
an error $O(1+BL)$, and it evaluates this integral, after the
substitution $x=Nu$, by the method of stationary phase at $u=1/2$,
with an error $O(N/B)$. The value $B=L$ is used in \cite{Robles} only
to verify the hypotheses of \cite[Theorem~1.1]{Robles} and to compare
the error term with the main term.
\end{proof}

\subsection{The exponent of $B$}\label{sec:displacement-exponent}

Take $B=L$ and $\alpha=B/N$. For large $N$ the hypotheses
\eqref{eq:aux-hyp} hold with $a/r=0/1$ and $\beta=B/N$, and the
quantity $B$ of Theorem~\ref{thm:aux} is $L$. By
Proposition~\ref{prop:stationary},
$|S_{f_B}(N,\alpha)|\sim N/\sqrt{2L}$, whereas
$N/L+N/L^{1/2+\eps}=o(N/\sqrt L)$ for every fixed $\eps>0$. Hence the
factor $B^{-1/2}$ in \eqref{eq:aux-main} cannot be replaced by
$B^{-1/2-\eps}$. This is \cite[Proposition~8.1]{Robles}.

\subsection{The range of $q$ and the term $\sqrt{Nq}$}\label{sec:upper-scale}

Fix $0<\eta<2$, and put $q=\lfloor N/L^{2-\eta}\rfloor$, $\alpha=1/q$
and $B=N/q$, so that $\alpha=B/N$ and $B\sim L^{2-\eta}$. By
Proposition~\ref{prop:stationary},
\[
 |S_{f_B}(N,1/q)|\sim\frac N{\sqrt{2B}}=\sqrt{\frac{Nq}2}
 \sim\frac N{\sqrt2\,L^{1-\eta/2}},
\]
whereas $N/L+N/\sqrt q\asymp N/L$. Thus \eqref{eq:folklore} does not
hold for all $q\le N/L^{2-\eta}$, and for these $q$ the term
$\sqrt{Nq}$ in \eqref{eq:three-term} cannot be replaced by a function
of smaller order. More generally, if $\lambda=\lambda(N)\to\infty$ and
$\lambda=o(L^2)$, the choice $q=\lfloor N\lambda/L^2\rfloor$ gives
$|S_{f_B}(N,1/q)|\asymp\sqrt\lambda\,N/L$. Hence the range
$q\le N/L^2$ in \eqref{eq:folklore} cannot be enlarged by a factor
$\xi(N)$ tending to infinity: take $\lambda=\min\{\xi(N),L\}$. See also the construction with Farey fractions in
\cite[Section~7(iii)]{MV}.

\subsection{The restriction on $\beta$}\label{sec:width-obstruction}

Let $\chi$ be the primitive character modulo $3$, let $N\ge6$ be a
multiple of $3$, and take $\alpha=1/3$ with $a/q=1/2$, so that
$|\beta|=1/6\le q^{-2}$. By \eqref{eq:character-obstruction},
$|S_\chi(N,1/3)|=N/\sqrt3$. On the other hand
$q\max\{1,N|\beta|\}=N/3$, so that the right side of
\eqref{eq:refined} is $\ll N/L+\sqrt N$. Thus \eqref{eq:refined} fails
if \eqref{eq:wider-local} is replaced by $|\beta|\le q^{-2}$ and an
upper bound for $q$. Here $D=N/3$.

\subsection{The density of the coefficients}\label{sec:density}

Bachman \cite[(0.7)]{Bachman} conjectured that, for every fixed
$0<\lambda\le1$,
\begin{equation}\label{eq:density-conj}
 |S_f(N,\alpha)|\ll_\lambda\frac NL+\frac1{\sqrt R}\sum_{n\le N}|f(n)|
\end{equation}
under the hypotheses of \eqref{eq:global} with $R\ge3$, uniformly for
the functions $f\in\mathcal B_1$ such that
\begin{equation}\label{eq:density-condition}
 \sum_{p\le N}\frac{|f(p)|}p\ge\lambda\log\log N;
\end{equation}
see also \cite[Conjecture~1.1]{MS}. This does not hold as stated. Let
$q=\prod_{3\le p\le y}p$ and $H_q=q/\varphi(q)$, let $\chi$ be a
primitive character modulo $q$, for example the product of the Legendre
symbols modulo the primes dividing $q$, and let $N$ be a multiple of
$q$ with $L\ge q$. By \eqref{eq:character-obstruction},
$|S_\chi(N,1/q)|=N/\sqrt q$, while $\sum_{n\le N}|\chi(n)|=N/H_q$.
Since $\log H_q=-\sum_{p\mid q}\log(1-1/p)=\sum_{p\mid q}1/p+O(1)$,
Mertens' theorem \cite[Theorem~2.7(d)]{MV2} gives
\[
 \sum_{p\le N}\frac{|\chi(p)|}p
 =\sum_{p\le N}\frac1p-\sum_{p\mid q}\frac1p
 =\log L-\log H_q+O(1).
\]
Here $H_q\ll\log\log(3q)$ by \cite[(2.3)]{Robles} and $L\ge q$, so that
the right side is at least $\frac12\log L$ when $y$ is large; thus
$\chi$ satisfies \eqref{eq:density-condition} with $\lambda=1/2$. Take
$\alpha=1/q$ and $R=q$; then $3\le R\le q\le N/R$. The bound
\eqref{eq:density-conj} would give
\[
 \frac N{\sqrt q}\ll\frac NL+\frac N{H_q\sqrt q},
\]
which is false for large $y$, because $\sqrt q/L\le q^{-1/2}$ and
$H_q\to\infty$. In this example the sum exceeds the second term of
\eqref{eq:density-conj} by the factor $H_q$; the same factor
$q/\varphi(q)$ occurs in the lower bound (ii)$'$ of
\cite[Section~1]{Bachman}.

\subsection{The leading constants}\label{sec:constant}

For a reduced fraction $a/q$ with $|\alpha-a/q|\le1/(qN)$ and
$q\le L^2$, Granville and Lamzouri \cite[(1.3)]{GL} conjecture that
$|S_f(N,\alpha)|\le(1+o(1))N/\sqrt q$ for every multiplicative
$f\in\mathcal B_1$. For functions that are not completely
multiplicative, the constant $1$ has to be replaced by a constant not
less than $4/\sqrt6$. Let $g$ be the multiplicative function with
$g(2^k)=g(3^k)=-1$ and $g(p^k)=1$ for primes $p\ge5$ and $k\ge1$. It has
period $6$, and its values at $1,\dots,6$ are $1,-1,-1,-1,1,1$. Put
$\zeta=\e(1/6)$, so that $\zeta^3=-1$ and $\zeta^2=\zeta-1$. Then
\[
 \sum_{n=1}^6g(n)\e(n/6)
 =\zeta-\zeta^2-\zeta^3-\zeta^4+\zeta^5+\zeta^6
 =2(1+\zeta-\zeta^2)=4,
\]
and therefore, for every multiple $N$ of $6$,
\[
 S_g(N,1/6)=\frac{2N}3=\frac4{\sqrt6}\cdot\frac N{\sqrt6}.
\]
The modulus need not be fixed. Let $\ell\ge5$ be prime, let $\chi_\ell$
be a primitive character modulo $\ell$, and put $q=6\ell$. By the
Chinese remainder theorem, for $(a,q)=1$ the sum of
$g(n)\chi_\ell(n)\e(an/q)$ over a period is the product of
$\sum_{b\bmod6}g(b)\e(a'b/6)$ and
$\sum_{c\bmod\ell}\chi_\ell(c)\e(a''c/\ell)$, where
$a'\equiv\pm1\pmod6$ and $(a'',\ell)=1$. The first sum equals $4$
because $g$ is real, and the second has modulus $\sqrt\ell$ by
\eqref{eq:induced-Gauss}. Hence
$|S_{g\chi_\ell}(N,a/q)|=(4/\sqrt6)N/\sqrt q$ for every multiple $N$ of
$q$, in particular for those with $q\le L^2$. The function
$g\chi_\ell$ is the one of \cite[Example~1.7]{GL},
where it is used for the sum with logarithmic weights; the constant
$4/\sqrt6$ is that of \cite[(1.6)]{GL} and of
Lemma~\ref{lem:Gauss-mass}.

For the sum with logarithmic weights, \cite[(1.6)]{GL} proposes the
bound $(4/\sqrt6+o(1))L/\sqrt q$ for $f\in\mathcal B_1$ and a reduced
fraction $a/q$ with $|\alpha-a/q|\le1/(qN)$, when
$q\le(L/\log L)^2$. This cannot hold in the whole of that range. Take
$f=1$, $\alpha=1/q$ and $q=\lfloor(L/\log L)^2\rfloor$, and let
$N\to\infty$. Then $\sqrt q\sim L/\log L$ and $\log q\sim2\log L$, so
that \eqref{eq:logq-sharpness} gives
\[
 \bigg|\sum_{n\le N}\frac{\e(n/q)}n\bigg|=\log q+O(1)
 \sim2\,\frac L{\sqrt q},
\]
and $2>4/\sqrt6$.

\section{The Montgomery-Vaughan class}\label{sec:MV-counter}

In this section we prove Theorem~\ref{thm:sharpness-summary} and the
assertions made after it. We write $H_q=q/\varphi(q)$, as in
Section~\ref{sec:gauss}.

\begin{proof}[Proof of Theorem~\ref{thm:sharpness-summary}]
Let $y\ge3$, and put
\[
 q=\prod_{3\le p\le y}p,\qquad\delta=\frac1{\log(2H_q)}.
\]
Then $H_q\ge3/2$ and $0<\delta<1$. Let $\chi$ be a primitive character
modulo $q$, for example the product of the Legendre symbols modulo the
primes dividing $q$. For $z>y$, let $\PP$ be the set of primes $p$ with
$y<p\le z$, and put
\begin{equation}\label{eq:counter-f}
 f_q(n)=\chi(n)(1+\delta)^{\omega_{\PP}(n)/2},\qquad
 \Pi_\delta=\prod_{p\in\PP}\Big(1+\frac\delta p\Big),
\end{equation}
where $\omega_{\PP}(n)$ is the number of primes of $\PP$ dividing $n$.
The function $f_q$ is multiplicative, and for $k\ge1$ we have
$|f_q(p^k)|=0$ if $p\mid q$, $|f_q(p^k)|=\sqrt{1+\delta}$ if
$p\in\PP$, and $|f_q(p^k)|=1$ otherwise. Since $H_q\to\infty$ as
$y\to\infty$, we have $\sqrt{1+\delta}\le A$ when $y$ is large in terms
of $A$. The parameters $y$ and $z$ are chosen below.

We first bound the mean square. Let $u(n)=|f_q(n)|^2$, so that
$u(n)=(1+\delta)^{\omega_{\PP}(n)}$ if $(n,q)=1$ and $u(n)=0$
otherwise. For $p\in\PP$,
\[
 \sum_{k\ge0}\frac{u(p^k)}{p^{ks}}=\frac{1+\delta p^{-s}}{1-p^{-s}},
 \qquad\text{so that}\qquad
 \Lambda_u(p^k)=\big(1+(-1)^{k-1}\delta^k\big)\log p.
\]
Moreover $\Lambda_u(p^k)=0$ if $p\mid q$, and $\Lambda_u(p^k)=\log p$
for the remaining primes. Hence $|\Lambda_u(n)|\le2\Lambda(n)$, that
is, $u\in\mathcal C(2)$. By \eqref{eq:Shiu} there are absolute
constants $x_0$ and $C$ such that
\[
 \sum_{n\le x}u(n)\le C\frac x{\log x}
 \exp\bigg(\sum_{p\le x}\frac{u(p)}p\bigg)\qquad(x\ge x_0).
\]
Let $y\ge x_0$ and $x>y$. Since $u(p)=0$ for $p\mid q$,
$u(p)=1+\delta$ for $p\in\PP$, and $u(p)=1$ for the other primes,
\[
 \sum_{p\le x}\frac{u(p)}p\le\sum_{p\le x}\frac1p-\sum_{p\mid q}\frac1p
 +\delta\sum_{p\in\PP}\frac1p.
\]
Here $\sum_{p\mid q}1/p=\log H_q+O(1)$ and
$\delta\sum_{p\in\PP}1/p=\log\Pi_\delta+O(1)$, because
$\sum_pp^{-2}<\infty$ and $\delta\le1$, and
$\sum_{p\le x}1/p=\log\log x+O(1)$ by Mertens' theorem
\cite[Theorem~2.7(d)]{MV2}. Therefore
\begin{equation}\label{eq:counter-mean}
 \sum_{n\le x}|f_q(n)|^2\le C_1x\frac{\Pi_\delta}{H_q}\qquad(x>y),
\end{equation}
where $C_1\ge1$ is an absolute constant; it does not depend on $y$,
$z$ or $A$. For $1\le x\le y$, the integers $n\le x$ coprime to $q$
are powers of $2$. They have no prime factor in $\PP$, and there are at
most $x$ of them, so that $\sum_{n\le x}|f_q(n)|^2\le x$.

We now choose $y$ so large that $y\ge x_0$, $\sqrt{1+\delta}\le A$ and
$H_q>2C_1$. Since $\sum_p1/p$ diverges, there is a least prime $z>y$
for which $\Pi_\delta\ge H_q/(2C_1)$. As $H_q/(2C_1)>1$ and each factor
of $\Pi_\delta$ is at most $2$, this $z$ satisfies
\begin{equation}\label{eq:product-stop}
 \frac{H_q}{2C_1}\le\Pi_\delta\le\frac{H_q}{C_1}.
\end{equation}
The bound \eqref{eq:counter-mean} and the case $x\le y$ then give
\begin{equation}\label{eq:all-x-verified}
 \sum_{n\le x}|f_q(n)|^2\le x\qquad\text{for every real }x\ge1.
\end{equation}
This proves \eqref{eq:counter-summary}, and $f_q\in\FA$.

It remains to prove \eqref{eq:counter-order}. Put
\[
 \eta=\sqrt{1+\delta}-1,\qquad q_1=q\prod_{p\in\PP}p,\qquad
 \Pi_\eta=\prod_{p\in\PP}\Big(1+\frac\eta p\Big).
\]
The function $f_q(n)\e(an/q)$ has period $q_1$, and the weight
$(1+\delta)^{\omega_{\PP}(n)/2}$ is the product over $p\in\PP$ of
factors equal to $\sqrt{1+\delta}$ if $p\mid n$ and to $1$ otherwise.
By the Chinese remainder theorem, the sum of $f_q(n)\e(an/q)$ over a
period is therefore the product of
$\sum_{m\bmod q}\chi(m)\e(am/q)=\overline{\chi(a)}\tau(\chi)$ and of
the numbers $p-1+\sqrt{1+\delta}=p(1+\eta/p)$ for $p\in\PP$. Since
$|\tau(\chi)|=\sqrt q$ by \eqref{eq:induced-Gauss}, we obtain, when $N$
is a multiple of $q_1$,
\begin{equation}\label{eq:CRT-exact}
 S_{f_q}(N,a/q)=\frac Nq\,\overline{\chi(a)}\tau(\chi)\,\Pi_\eta,
 \qquad|S_{f_q}(N,a/q)|=\frac N{\sqrt q}\,\Pi_\eta.
\end{equation}
We compare $\Pi_\eta$ with $\Pi_\delta$. Since $\delta=2\eta+\eta^2$,
we have $(1+\eta/p)^2\le1+\delta/p$, and so
$\Pi_\eta\le\Pi_\delta^{1/2}$. Put $\gamma=\eta/\delta
=1/(1+\sqrt{1+\delta})$. The concavity of the logarithm gives
$\log(1+\gamma\delta/p)\ge\gamma\log(1+\delta/p)$, and so
$\Pi_\eta\ge\Pi_\delta^\gamma$. Now
\[
 0\le\frac12-\gamma=\frac\delta{2(1+\sqrt{1+\delta})^2}\le\frac\delta8,
 \qquad0<\delta\log\Pi_\delta\le\delta\log(2H_q)=1,
\]
by \eqref{eq:product-stop}. Consequently
\begin{equation}\label{eq:product-comparison}
 e^{-1/8}\,\Pi_\delta^{1/2}\le\Pi_\eta\le\Pi_\delta^{1/2}.
\end{equation}
By \eqref{eq:product-stop}, \eqref{eq:CRT-exact},
\eqref{eq:product-comparison} and $H_q/q=1/\varphi(q)$, we obtain for
the multiples $N$ of $q_1$
\[
 \frac{e^{-1/8}}{\sqrt{2C_1}}\,\frac N{\sqrt{\varphi(q)}}
 \le|S_{f_q}(N,a/q)|\le\frac1{\sqrt{C_1}}\,\frac N{\sqrt{\varphi(q)}}.
\]

For arbitrary real $N\ge1$, write $N=kq_1+s$ with an integer $k\ge0$
and $0\le s<q_1$, and put
$\mu_a=\overline{\chi(a)}\tau(\chi)\Pi_\eta/q$. By periodicity,
$S_{f_q}(N,a/q)=kq_1\mu_a+S_{f_q}(s,a/q)$. The Cauchy-Schwarz
inequality and \eqref{eq:all-x-verified} give
$|S_{f_q}(s,a/q)|\le s$, and $|\mu_a|\le1$ by the upper bound just
proved. Hence
\begin{equation}\label{eq:counter-all-lengths}
 S_{f_q}(N,a/q)=N\mu_a+O(q_1),\qquad
 \frac{e^{-1/8}}{\sqrt{2C_1\varphi(q)}}\le|\mu_a|
 \le\frac1{\sqrt{C_1\varphi(q)}},
\end{equation}
where the error term has modulus at most $2q_1$. It is at most
$N|\mu_a|/2$ as soon as
$N\ge4e^{1/8}\sqrt{2C_1}\,q_1\sqrt{\varphi(q)}$, and for these $N$
\eqref{eq:counter-order} holds, with absolute constants and for every
$a$ coprime to $q$.

Finally, $H_q\asymp\log\log q$. Indeed,
\[
 H_q=\prod_{3\le p\le y}\Big(1-\frac1p\Big)^{-1}
 \ge\sum_{\substack{n\le y\\n\textup{ odd}}}\frac1n\gg\log y,
\]
while $\log q\le y\log y$ gives $\log\log q\ll\log y$; and
$H_q\ll\log\log(3q)$ by \cite[(2.3)]{Robles}. If $N$ is such that
$L\ge q$, then $q\le N/L^2$ and $N/L\le N/q$, as in
Section~\ref{sec:q-sharpness}.
\end{proof}

We now deduce the assertions made after
Theorem~\ref{thm:sharpness-summary}. Fix $A>1$, let $q$ and $f_q$ be as
in the proof, and let $N$ be so large that
\eqref{eq:counter-order} holds, $L\ge q$ and $N\ge L^4$. Then
$q\le L\le N/L^3$; in particular $q\le N/L^2$, and
$R=q$ satisfies $R\le q\le N/R$. By \eqref{eq:counter-order},
\[
 |S_{f_q}(N,a/q)|\asymp\sqrt{H_q}\,\frac N{\sqrt q},
 \qquad\text{whereas}\qquad\frac NL+\frac N{\sqrt q}\le\frac{2N}{\sqrt q}.
\]
Since $H_q\to\infty$ as $q\to\infty$, none of the following holds
uniformly in $\FA$: the bound \eqref{eq:folklore}; the bound
\eqref{eq:R-intro} with $r$ in place of $\varphi(r)$, here with $r=q$
and $\beta=0$; the bound with $N/\sqrt R$ in place of the second term
of \eqref{eq:MV-intro}. Since
$H_q\asymp\log\log q$, the sum has order
$NR^{-1/2}\sqrt{\log\log(3R)}$, so that this factor cannot be removed
in \cite[Corollary~1.1]{Robles}. Note also that
$\sup_n|f_q(n)|=(1+\delta)^{\#\PP/2}\ge\Pi_\eta$ tends to infinity
with $q$, so that no class $\BA$ contains all the functions $f_q$.

\section{Conclusion and future work}\label{sec:conclusion}

Theorems~\ref{thm:global} and~\ref{thm:refined} give the bounds
\eqref{eq:folklore-conj} and \eqref{eq:refined-conj} for all bounded
multiplicative functions, in the ranges $q\le N/L^2$ and
$q\max\{1,N|\beta|\}\le N/L^2$ respectively; complete multiplicativity
is not required. The examples of Section~\ref{sec:sharpness} show that
these bounds cannot be improved in several respects. In the class $\FA$
with $A>1$, Theorem~\ref{thm:sharpness-summary} shows that the totient
in \eqref{eq:R-intro} cannot be replaced by $r$.

The functions $f_q$ of Theorem~\ref{thm:sharpness-summary} have modulus
greater than $1$ at some primes, so that this theorem does not decide
whether Theorems~\ref{thm:global}-\ref{thm:refined} hold in
$\mathcal F_1$. This will be the subject of future research.

A second question is to take the density of the coefficients into
account when $\sum_{n\le N}|f(n)|$ is of smaller order than $N$. By
Section~\ref{sec:density}, the bound \eqref{eq:density-conj} proposed
for this purpose has to be modified.
Granville and Lamzouri \cite[(1.4)]{GL} state a conjecture of this kind
for the sum restricted to the integers $n\le N$ without prime factors
exceeding $y$. For this restricted sum, de la Bret\`eche
\cite[Proposition~1]{dB} proved an estimate for multiplicative $f$
satisfying the mean-square condition in \eqref{eq:FA}; see also
\cite[Section~10.2]{dBG}.

Finally, the estimates of this paper concern the order of magnitude
only. Granville and Lamzouri propose leading constants for
$S_f(N,\alpha)$ and for the sum with logarithmic weights
\cite[(1.3) and (1.6)]{GL}. By Section~\ref{sec:constant}, the first
has to be at least $4/\sqrt6$ for multiplicative $f\in\mathcal B_1$,
and the second cannot be $4/\sqrt6$ in the whole range
$q\le(L/\log L)^2$. The constants, and the ranges of $q$ in which they
hold, remain to be determined.

\section*{Acknowledgement}

The author is grateful to Dirk Zeindler and Alexandru Zaharescu for
helpful discussions. The author conceived the paper, which builds on
the method of \cite{Robles}, developed the overall strategy of the
proof, and wrote the first draft.

In checking parts of the arguments, and in revising the
exposition, the author used ChatGPT Astra~6. The mathematical verification included
the reduction to $A=1$ in Lemma~\ref{lem:linear-A}, the deductions
involving rational approximation and logarithmic weights in
Sections~\ref{sec:deductions} and~\ref{sec:log-weights}, and the examination
of the maximal mean-square estimates over reduced residue classes in
Section~\ref{sec:energy}. In checking the proofs, particular attention
was given to the use of a single exceptional set throughout the required
intervals, the passage from equally spaced points to the full maximum
in Lemma~\ref{lem:coarse-grid}, and the mean-square condition
\eqref{eq:all-x-verified} for every initial segment in the counterexample.

Subsequently, Claude Fable~5.1 was used for an adversarial audit of the
final draft and for further revision of the exposition. The author takes full
responsibility for the mathematical content and the final text.

\end{document}